\documentclass{article}

\usepackage[utf8]{inputenc}
\usepackage{amsmath}
\usepackage[english]{babel}
\usepackage{newlfont}
\usepackage{latexsym}
\usepackage[utf8]{inputenc}
\usepackage{enumerate}
\usepackage{xy}
\usepackage{amsthm,amssymb,amsmath,amsopn,amsfonts,pb-diagram,accents}
\usepackage{cite}
\usepackage{stmaryrd,calc,enumerate,mathrsfs,amscd,wasysym,footnote}
\usepackage{graphicx}
\usepackage{fancyhdr}
\usepackage{indentfirst}
\usepackage{geometry}
\usepackage{mathtools}
\usepackage{bbm}
\usepackage{float}
\usepackage{hyperref}
\usepackage{comment}
\usepackage{forest}
\usepackage{tikz}
\usepackage{wrapfig}
\usepackage{xcolor}
\usepackage{algorithm}
\usepackage[noend]{algpseudocode}
\usepackage[square, numbers]{natbib}
\usepackage{graphicx}
\usepackage{cleveref}
\usepackage{tabularx}
\usepackage{soul}
 \usepackage{titlefoot}

\usepackage[para]{footmisc} %per footnotes stessa riga
\usepackage{fixfoot} %per footnotes ripetute
\DeclareFixedFootnote{\repnote}{Dipartimento di Scienze Matematiche “Giuseppe Luigi Lagrange”, Dipartimento di Eccellenza 2018-2022, Politecnico di Torino, Corso Duca degli Abruzzi 24, 10129 Torino Italy.}

\theoremstyle{plain}

\newtheorem{thm}{Theorem}[section]

\newtheorem{cor}[thm]{Corollary}
\newtheorem{lem}[thm]{Lemma}
\newtheorem{prop}[thm]{Proposition}

\theoremstyle{definition}
\newtheorem{defn}[thm]{Definition} 
\newtheorem{oss}[thm]{Remark}

\theoremstyle{remark}

\DeclareMathOperator{\argmin}{argmin}
\DeclareMathOperator{\argmax}{argmax}
\DeclareMathOperator{\di}{d}

\DeclarePairedDelimiter\floor{\lfloor}{\rfloor}
\DeclarePairedDelimiter\ceil{\lceil}{\rceil}
\usepackage[affil-it]{authblk}

\author{Matteo Levi$^1$\footnote{matteo.levi@polito.it}, Federico Santagati$^1$\footnote{federico.santagati@polito.it},  Anita Tabacco$^1$\footnote{anita.tabacco@polito.it} and Maria Vallarino$^1$\footnote{maria.vallarino@polito.it}}
\title{{\bf Analysis on trees with nondoubling flow measures}}
\date{\vspace{-5ex}}
\begin{document}

\unmarkedfntext{{\em 2020 Mathematics Subject Classification}: 05C05; 05C21; 30H10; 30H35; 43A99.\newline\hspace*{0.5em}{\em Keywords}: trees; nondoubling measure; Calder\'on--Zygmund theory; Hardy spaces; BMO spaces.\newline\hspace*{0.5em}$^1$: Dipartimento di Scienze Matematiche “Giuseppe Luigi Lagrange”, Dipartimento di Eccellenza 2018-2022, Politecnico di Torino, Corso Duca degli Abruzzi 24, 10129 Torino Italy.\newline}

\maketitle

  \begin{abstract}
 We consider trees with root at infinity endowed with {\emph{flow measures}}, which are nondoubling measures of at least exponential growth and which do not satisfy the isoperimetric inequality. In this setting, we develop a Calderón--Zygmund theory  and we define $BMO$ and Hardy spaces, proving a number of desired results extending the corresponding theory as known in more classical settings. 
\end{abstract}

%{\small\textbf{\textit{Keywords---}} trees; nondoubling measure; Calder\'on--Zygmund theory; Hardy spaces; $BMO$ spaces.}

%{\small\textbf{\textit{Mathematics Subject Classification 2020 ---}} 05C05; 05C21; 30H10; 30H35; 43A99.}

%{\small\tableofcontents}

\section{Introduction}

Let $T=(V,E)$ be a tree, where $V$ is the set of vertices and $E$ the set of edges. 
We do not put any a priori restriction on the combinatorics of the tree, in particular we do not require $T$ to be homogeneous. We fix an origin $o\in V$, a geodesic $g$ emanating from $o$ and the corresponding point $\zeta_g$ in the natural boundary $\partial T$ of the tree. This boundary point serves as a root of $T$, and induces a level structure on the tree. In this paper we consider \textit{flow measures} on the set of vertices $V$, i.e., functions $m:V\rightarrow \mathbb R^+$ whose value at each point can be reconstructed as the sum of the values the function takes at the \textit{sons} of that point. See Section \ref{sec: pr} for all the details.

The interest in such measures comes from the fact that there is a natural injective lifting procedure to map Borel measures on the boundary of the tree to measures on its vertices, and flow measures are exactly the range of such a lifting operator. Moreover, $p-$harmonic functions on trees can be characterized as appropriate potentials of flow measures (see \cite{CL}). The word flow is commonly used in Network Theory as such measures are well suited to model physical phenomena such as the circulation of traffic in the cities, of current in electrical circuits, of water in pipes et cetera. The conservation property characterizing flows is equivalent to the \textit{Kirchhoff's current law}: the total current received by a vertex must equal the total current leaving the vertex. The existence of (finite energy) flows on a graph has been proved to have outstanding theoretical importance, being intimately related to the transience and recurrence properties of the graph itself; for a thorough account on this perspective we refer to the beautiful book \cite{LP}.

The metric measure space $(T,d,m)$, where $T$ is a general infinite tree, $d$ the geodesic distance and $m$ a flow measure, is a concrete example of nondoubling metric space without growth restrictions on the measure. Our goal is to investigate some pieces of theory coming from classical Harmonic Analysis in this context; in particular, we study the behaviour of a Hardy-Littlewood maximal function, we describe a Calder\'on--Zygmund decomposition of integrable functions and we introduce Hardy and $BMO$ spaces on $(T,d,m)$, when $m$ is locally doubling.

\smallskip

The classical Calder\'on--Zygmund theory and the standard theory of Hardy and $BMO$ spaces \cite{CW, FS, S} were introduced in $(\mathbb R^n,d,\lambda)$, where $d$ is the 
Euclidean metric and $\lambda$ denotes the Lebesgue measure and more generally on spaces of homogeneous type, namely metric measure spaces $(X,d,\mu)$ where the 
doubling condition is satisfied, i.e., there exists
a constant $C$ such that
\begin{equation}\label{doubling}
\mu\bigl(B_{2r}(x)\bigr)
\leq C\, \mu\bigl(B_r(x)\bigr)
\qquad\forall x \in X\,, \quad\forall r >0,
\end{equation}
where $B_r(x)$ denotes the ball centred at $x$ of radius $r$. 

It is worth noticing that in the setting of (possibly weighted) graphs with the doubling property new Hardy and $BMO$ spaces associated with a discrete Laplacian were introduced in \cite{bui, buiduong, fe}; various characterizations of such spaces and applications to singular integrals were obtained.  

Extensions of the theory of singular integrals and of the Hardy and $BMO$ spaces have been considered also on various metric measure spaces not satisfying the doubling condition \eqref{doubling} but fulfilling some other measure growth assumption (see, e.g., \cite{CMM, MMV, MOV, NTV, T, To, V, Ve})\ %\Comment{in the case when the measure is of polynomial growth and \cite{CMM, MMV, T, V} in the case when the measure grows exponentially).} 
or a geometric condition (see \cite{arthur}). We remark that the locally doubling flow measures we consider can grow at arbitrarily large rate.%{\color{red} 
%Io citerei anche un articolo di Artur su Duke "$BMO$FOR NONDOUBLING MEASURES": in questo lavoro considerano $BMO$e $H^1$ per misure nondoubling senza condizioni di crescita, se non sbaglio, ma solo richiedendo una condizione geometrica abbastanza leggera, ovvero che la misura si annulli su iperpiani paralleli agli assi coordinati.}

Carbonaro, Mauceri and Meda \cite{CMM} defined a Hardy space adapted to any metric measure space which satisfies some geometric assumption, namely the local doubling  property, the isoperimetric property and the approximate midpoint property. Their theory does not apply to the spaces we consider since they do not satisfy the isoperimetric property.

We also mention that some results on homogeneous trees endowed with the counting measure, which is not a flow measure, have been obtained in the literature. More precisely, Naor and Tao \cite{tao} studied the boundedness of the Hardy--Littlewood maximal function with respect to the family of balls. The boundedness of singular integrals associated with the combinatorial Laplacian has been investigated in \cite{CMS}, while Celotto and Meda \cite{CM} studied various Hardy spaces in this context. In \cite{KPT} the authors introduced Hardy spaces of harmonic functions on nonhomogeneous trees; their theory  can be thought as an analog of the classical theory of Hardy spaces on the unit disc.

%Their theory is useful to study the boundedness of singular integral operators related to the standard Laplacian defined on trees which is self-adjoint with respect to the counting measure; the theory we develop here instead is useful to study singular integral operators related to a distinguished Laplacian self-adjoint on $L^2(\mu)$ (see Subsection \ref{SubSecSingInt}). 

\smallskip

This paper generalizes in various directions the results in \cite{hs}, \cite{arditti} and \cite{ATV1} {concerning} 
a Calder\'on--Zygmund theory, Hardy and $BMO$ spaces on homogeneous trees endowed with a distinguished measure.  First of all, we consider general non-regular trees. Moreover, we consider a class of measures, i.e., the locally doubling flows, among which the measure considered in \cite{hs, arditti, ATV1} is a particular example. The definition of the family of admissible sets, which is a key ingredient to develop all the theory of the present paper, is strongly inspired by the one given by Hebisch and Steger in \cite{hs}, but is more general, even in their setting. Indeed, the shape of our sets is less rigid and this allows to obtain suitable decomposition and expansion algorithms which were not available in the setting of \cite{hs}. These algorithms are also a fundamental tool to prove some results involving Hardy and $BMO$ spaces which were missing in \cite{arditti, ATV1}.

\smallskip

The paper is organized as follows. In Section \ref{sec: pr} we introduce the notation and characterize the properties of being locally doubling, doubling and of exponential growth for flow measures. Moreover, we show that these measures never satisfy the isoperimetric property. We then restrict our attention to a tree $T$ endowed with a locally doubling flow measure $m$; this implies that $T$ is of uniformly bounded degree. In such a case, we introduce a family of subsets of $V$ called {\emph{admissible trapezoids}} (Section \ref{s: admtr}), which exhibits interesting geometric properties.  In particular, the Hardy--Littlewood maximal operator associated with such a family turns out to be of weak type $(1,1)$ and bounded on $L^p(m)$ for every $p\in (1,\infty]$. Moreover, we show that every integrable function admits a Calder\'on--Zygmund decomposition where the admissible trapezoids are involved. In Section \ref{s: BMOHardy} we introduce a Hardy space $H^1(m)$ and a space $BMO(m)$ related with the family of admissible trapezoids;  we prove that $BMO(m)$ can be identified with the dual of $H^1(m)$ and that a John--Nirenberg inequality holds in this setting. Finally, in Section \ref{s: int}, we show that such spaces satisfy good interpolation properties, both with respect to the real and the complex interpolation methods, so that they  can be used for endpoint boundedness results for integral operators. Applications of this theory to the study of boundedness of integral operators, like Riesz transforms  and spectral multipliers of suitable Laplacians, is work in progress.

\smallskip

Along the paper, constants appearing in different inequalities are sometimes related to each others. To highlight such connections we prefer to provide explicitly most of the constants. However, when the exact values are unimportant for us, we use the standard notation $A(x)\lesssim B(x)$ to indicate that there exists a positive constant $c$, independent from the variable $x$ but possibly depending on some involved parameters, such that $A(x)\leq c B(x)$ for every $x$.

\section{Preliminaries}\label{sec: pr}

\subsection{Trees with root at infinity}\label{subs: pr}

An unoriented graph is a vertex set $V$ endowed with a symmetric relation $\sim$. If $x\sim y$ we say there is an edge connecting $x$ to $y$, which we identify with the one connecting $y$ to $x$. The set of edges is denoted by $E$. Let $q(x)=|\lbrace y\in V: \ y\sim x\rbrace|-1$ denote the number of neighbours of $x$ (minus 1, for convenience of notation). We say that the graph is of \textit{bounded degree} if there exists a universal positive constant $q$ such that $q(x)\leq q$ for all $x\in V$. Consider a sequence of vertices $\lbrace x_j\rbrace$ such that $x_j\sim x_{j+1}$. This naturally identifies an associated sequence of edges $\lbrace e_j\rbrace$, where $e_j$ is the edge connecting $x_j$ to $x_{j+1}$. We say that $\lbrace x_j\rbrace$ is a \textit{path} if $\lbrace e_j\rbrace$ does not contain repeated edges. If the path $\gamma=\lbrace x_j\rbrace_{j=0}^n$ is finite, $x_0$ and $x_n$ are called the endpoints of $\gamma$. The geodesic distance $d(x,y)$ counts the minimum number of edges one has to cross while moving from $x$ to $y$ along a path. Any path realizing such a distance for every couple of vertices belonging to it is called a geodesic. We denote by $\Gamma$ the family of geodesics.

A graph is connected if every couple of vertices belongs to a path. Connected graphs having no paths with repeated vertices are called trees. In particular the relation $\sim$ is never transitive on a tree - there are no triangles. Also, it is clear that trees are uniquely geodesic spaces: for every couple of vertices $x,y$ in a tree $T$, there exists a unique path (which is necessarily a geodesic) having $x$ and $y$ as endpoints. Hence, without risk of confusion, we also denote by $[x,y]$ such a geodesic.

Let $T=(V,E)$ be a tree. We fix a distinguished point $o\in V$ which we call the \textit{origin} of the tree. We write $\Gamma_0$ for the family of half infinite geodesics having an endpoint in the origin, $\Gamma_0=\lbrace \gamma=\lbrace x_j\rbrace_{j=0}^\infty\in \Gamma, x_0=o\rbrace$. The \textit{boundary} of the tree $\partial T$ is classically identified with the set of labels corresponding to elements of $\Gamma_0$,
\begin{equation*}
    \partial T=\lbrace \zeta_\gamma: \ \gamma\in \Gamma_0\rbrace.
\end{equation*}
A point $z\in\overline{V}=V\cup \partial T$ can be chosen to play the role of \textit{root} of the tree. The role of such a point is to induce a partial order relation on $\overline{V}$, or more intuitively, to act as a base point for a radial foliation of the tree. We say that $x\geq y$ if and only if $x\in [z,y]$. We define the projection of $x$ on the geodesic $[o,z]$ as
\begin{equation*}
    \Pi_z(x)=\argmin_{y\in[o,z]} d(x,y),
\end{equation*}
and the \textit{level} of $x$ as
\begin{equation*}
    \ell(x)=d(o,\Pi_z(x))-d(\Pi_z(x),x).
\end{equation*}
In levels notation the rule for the order relation says that $x\geq y$ if and only if $\ell(x)-\ell(y)=d(x,y)$. Observe that if $x\leq o$, then $\ell(x)=-d(x,o)$. In particular, if one chooses the root to coincide with the origin, then the level of a point is just (minus) its distance from the origin, i.e., its radial coordinate, and the tree can be interpreted as a model for the unit disc $\mathbb{D}$. In this paper, however, we decide to fix the root as a point $\zeta_g\in \partial T$, $g$ being a distinguished half infinite geodesic starting at the origin. With this choice, the geometric interpretation of a level set in the unit disc is not so neat anymore. Instead, it is helpful  to switch to a half-plane model point of view; in analogy to the conformal trasformation of the unit disc onto the upper-half plane, mapping $\partial \mathbb{D}\setminus \lbrace\zeta\rbrace$ to $\mathbb{R}$ and $\zeta$ to the point at infinity, we can interpret the tree rooted at $\zeta_g$ as a conformal copy of the one rooted at the origin, and its boundary as a representation of the Riemann sphere. Following this point of view, hereinafter we will write $\partial T$ for $\lbrace \zeta_\gamma\in \Gamma_0\setminus\lbrace\zeta_g\rbrace\rbrace$ and interpret $\zeta_g$ as a separate special point (the point at infinity). It is easily seen that, with the upper-half plane model in mind, a level set plays the role of a line parallel to the real axis, which in the disc model would be an horocycle tangent to the boundary point $\zeta_g$, and the level of a point plays the role of $y-$coordinate in the parametrization of the tree.

We define some further  notation that will be useful in the treatment. The \textit{predecessor} of $x$ is the unique vertex $p(x)$ such that $x\sim p(x)$ and $\ell(p(x))=\ell(x)+1$, while $y$ is a \textit{son} of $x$ if it belongs to the set $s(x)=\lbrace y\sim x: \ \ell(y)=\ell(x)-1\rbrace$. Observe that $|s(x)|=q(x)$.
We define the \textit{confluent} of $x,y\in \overline{V}$ to be the point
\begin{equation*}
\begin{split}
    x\wedge y&=\argmax\lbrace \ell(z): \ z\in [x,y]\rbrace=\argmin\lbrace \ell(z): \ z\geq x, z\geq y\rbrace.
\end{split}
\end{equation*}
Observe that the level function can be written as $\ell(x)=d(x\wedge o,o)-d(x\wedge o,x)$. The \textit{tent} rooted in $x$ is the set  $V_x=\lbrace y\in V: \ y\leq x\rbrace$ and we denote by $\partial T_x$ its boundary, $\partial T_x=\lbrace\zeta\in \partial T: \ \zeta\leq x\rbrace=\lbrace\zeta\in \partial T: \ \zeta\wedge x=x\rbrace$. The family $\lbrace \partial T_x\rbrace_{x\in V}$ can be used as a basis for the topology on $\partial T$. Borel measures on the boundary can then be considered, accordingly.

Finally, we introduce the following convenient notation which will be widely used throughout the paper: whenever we fix a vertex $x_0$, we denote by $x_k$ its $k-$th predecessor, namely $x_k=p(x_{k-1})$ for any integer $k\geq 1$. 
Clearly $x_k=x_k(x_0)$ depends on $x_0$, but since the basis point $x_0$ will always be clear from the context we will simply write $x_k$.

\subsection{Flow measures}

From now on, $T$ will always denote a tree rooted at $\zeta_g$, endowed with the level structure described in the previous section, and $V$ the set of its vertices. Since $V$ is a discrete set, every function on $V$ defines a measure. With some abuse of notation, given a function $m:V\to\mathbb{R}$ we denote by $m$ also the associated measure, given by
\begin{equation*}
    m(A)=\sum_{x \in A}m(x), \quad A\subseteq V. %=\sum_{x \in V}\mathbb{I}_A(x)m(x)
\end{equation*} Given a function $f:V \to \mathbb{C}$ and a set $A \subset V$, we define $\displaystyle\int_A f \ dm = \sum_{y \in A} f(y) m(y)$.\\We say that $m$ is a \textit{flow} if,  $\forall x\in V$, it holds
\begin{equation*}
    m(x)=\sum_{y\in s(x)}m(y).
\end{equation*}
Flow measures on $V$ are special in the fact that they are in a 1-1 relation with Borel measures on the boundary of the tree.
More precisely, any flow measure $m$ can be extended to a measure on $\partial T$ through the correspondence
\begin{equation}\label{measure boundary}
    m(\partial T_x)=m(x),
\end{equation}
and conversely, if $m$ is a non-negative Borel measure on $\partial T$, then the function $m:V\to \mathbb{R}$ defined by \eqref{measure boundary} is a flow (by the additivity property of measures). Observe that, if a non-negative flow vanishes at a point $x$, then it vanishes on the whole tent $V_x$. Throughout the paper, flow measures will be implicitly assumed to be strictly positive.

We are interested in the relation between flow measures and the doubling property. Let $S_r(x_0)=\lbrace x\in V: \ d(x,x_0)=r\rbrace$ and $B_r(x_0)=\lbrace x\in V: \ d(x,x_0)\leq r\rbrace$ be, respectively, the sphere and the ball of radius $r\in \mathbb{N}$ centered at $x_0\in V$ with respect to the geodesic distance metric. We say that $m$ is a \textit{locally doubling measure} on $V$ if, for every $r>0$, there exists a constant $C_r> 1$ such that
\begin{equation}\label{locally doubling}
    m(B_{2r}(x_0))\leq C_r m(B_r(x_0)), \quad \forall x_0\in V.
\end{equation}
If there exists a universal constant $C> 1$ such that for every $r>0$, relation \eqref{locally doubling} holds with $C_r=C$, then the measure $m$ is said (globally) \textit{doubling}.

Next technical lemma provides explicit expressions for the mass of spheres and balls for a general flow measure and useful upper and lower bounds for the ratio between measures of balls. 

%The next technical lemma provides useful upper and lower bounds for the ratio between measures of balls. The estimates in the statement of the lemma are rough, but they are enough for our scopes. Finer estimates can be found along the proof, as well as explicit expressions for the mass of spheres \eqref{mass sphere} and balls \eqref{mass ball} for a general flow measure, with respect to the counting metric.
\begin{lem}\label{lem}
Let $m$ be a flow measure. Fix $x_0\in V$ and, for $k\geq 0$, let $x_{k+1}=p(x_k)$. For every $r\in \mathbb{N}$, it holds
\begin{itemize}
    \item[(i)]  $m(S_r(x_0))=m(x_{r-1})+m(x_r)$;
    \item[(ii)] $ m(B_r(x_0))=2\displaystyle\sum_{j=0}^{r-1}m(x_j)+m(x_r)$;
    \item[(iii)] $\displaystyle\frac{m(x_{2r})}{(2r+1)m(x_r)}\leq \frac{m(B_{2r}(x_0))}{m(B_r(x_0))} \leq\frac{(4r+1)m(x_{2r})}{m(x_r)}$.
\end{itemize}

\end{lem}
\begin{proof}
Define the $l-$level slice of the sphere $S_r(x_0)$ as $S_r^l(x_0)=S_r(x_0)\cap\lbrace x\in V: \ \ell(x)=l\rbrace$, and set $l(k)=\ell(x_0)-r+2k$. Then 
\begin{equation*}
    S_r(x_0)=\bigcup_{k=0}^{r}S_r^{l(k)}(x_0).
\end{equation*}
%Exploiting the fact that for any $x\in V$
%\begin{equation*}
 %   d(x_0,x)=d(x_0,x_0\wedge x)+d(x_0\wedge x,x)=2l(x_0\wedge x)-l(x_0)-l(x),
%\end{equation*}
%for $k=0,\dots,r$, we have
It's easily seen that
\begin{align*}
&S^{l(0)}_r(x_0)=\{x \in V \ : x \le x_0, \ell(x)=l(0)\}, \\ 
    &S_r^{l(k)}(x_0)%&=\lbrace x\in V: \ l(x_0\wedge x)=l(x_0)+k, l(x)=l(k)\rbrace\\
    %&=\lbrace x\in V: \ x_0\wedge x=x_k, l(x)=l(k)\rbrace\\
    =\lbrace x\in V: \ x\leq x_k, \ell(x)=l(k)\rbrace\setminus \lbrace x\in V: \ x\leq x_{k-1}\rbrace, \qquad \text{for} \  k>1.
\end{align*}
Observe that $m(S_r^{l(0)}(x_0))=m(x_0)$, $S_r^{l(r)}(x_0)=\{x_r\}$ and, for $1\leq k \leq r-1$, $S_r^{l(k)}(x_0)\neq\emptyset$ if and only if $q(x_k)\geq 2$. If $m$ is a flow measure, then
\begin{equation*}
    m(S_r^{l(k)}(x_0))=m(s(x_k))-m(s(x_{k-1}))= m(x_k)-m(x_{k-1}), \quad 1\leq k \leq r-1,
\end{equation*}
which equals zero if $q(x_k)=1$, as expected. The flow measure of the sphere, for $r\geq 1$, is then given by
\begin{equation}\label{mass sphere}
    m(S_r(x_0))= \sum_{k=0}^r m(S_r^{l(k)}(x_0))=m(x_0)+m(x_r)+\sum_{k=1}^{r-1} m(x_k)-m(x_{k-1})
    =m(x_{r-1})+m(x_r).
\end{equation}
We can now derive the flow measure of the ball considering its foliation by means of spheres,
\begin{equation}\label{mass ball}
    m(B_r(x_0))=\sum_{j=0}^r m(S_j(x_0))= m(x_0)+\sum_{j=1}^r m(x_{j-1})+m(x_j)=2\sum_{j=0}^{r-1}m(x_j)+m(x_r).
\end{equation}
Clearly $m(B_r(x_0))\geq \sum_{j=0}^rm(x_j)$, and we derive
\begin{equation*}
       \frac{m(x_{2r})}{(2r+1)m(x_r)}\leq \frac{\displaystyle \sum_{j=0}^{2r}m(x_j)}{2\displaystyle \sum_{j=0}^{r-1}m(x_j)+m(x_r)} \leq \frac{m(B_{2r}(x_0))}{m(B_r(x_0))} \leq \frac{2\displaystyle \sum_{j=0}^{2r-1}m(x_j)+m(x_{2r})}{\displaystyle \sum_{j=0}^r m(x_j)}\leq\frac{(4r+1)m(x_{2r})}{m(x_r)}.
\end{equation*}
\end{proof}
Next proposition  gives two properties which are equivalent to the locally doubling condition and more convenient to use in practice.
\begin{prop}\label{ldc}
Let $m$ be a flow measure. The following  are equivalent.
\begin{itemize}
    \item[(i)] The measure $m$ is locally doubling.
    \item[(ii)] There exists a constant $c> 1$ such that
                \begin{align}\label{control sons}
                    &m(x)\le c m(y), \ \ \ \ \qquad \forall x\in V, y\in s(x), \\ 
                 \label{lower bound}
                                     & m(x) \ge \frac{c}{c-1}m(y), \  \quad \forall x \in V \ with \ q(x) \ge 2, y \in s(x).
                \end{align}
\end{itemize}
\end{prop}
\begin{proof}
 Assume \eqref{control sons} holds. Then for any $x_0\in V$, $r>0$, from Lemma \ref{lem}  we have
\begin{equation*}
        \frac{m(B_{2r}(x_0))}{m(B_r(x_0))}\leq (4r+1)c^r=C_r,
\end{equation*}
hence $m$ is locally doubling. \\ 
Conversely, let $m$ be a locally doubling flow. Then, for every $x\in V$, $y\in s(x)$ and $z\in s(y)$, again from \eqref{locally doubling} and Lemma \ref{lem},
\begin{equation}\label{star}
        C_1\geq\frac{m(B_{2}(z))}{m(B_1(z))}\geq\frac{m(x)}{3m(y)}.
\end{equation}
If \eqref{control sons} would not hold, we could find two sequences of vertices $\lbrace x _j\rbrace$ and $\lbrace y _j\rbrace$, with $y_j\in s(x_j)$, such that $\limsup_j m(x_j)/m(y_j)=+\infty$, contradicting \eqref{star}.
Moreover,  inequality \eqref{control sons} implies that
\begin{equation*}
        m(x)=\sum_{y\in s(x)}m(y)\geq\frac{1}{c}\sum_{y\in s(x)}m(x)=\frac{q(x)m(x)}{c}.
\end{equation*} Hence, $q(x) \le c$ for all $x \in V$.  
Then, for $x\in V$ with $q(x)\geq 2$ we have
\begin{equation*}
    m(x)=m(y)+\sum_{z\in s(x)\setminus \lbrace y\rbrace}m(z)\geq m(y)+\frac{q(x)-1}{c}m(x),
\end{equation*}
from which it follows that
\begin{equation*}
    m(x)\geq \frac{c}{c-q(x)+1}m(y)\geq \frac{c}{c-1}m(y).
\end{equation*}
This completes the proof.
%Conversely, assume \eqref{lower bound} holds. Relation \eqref{control sons} is trivially true if $q(x)=1$, while for $q(x)\geq 2$ we have,
%\begin{equation*}
 %   m(x)=m(y)+\sum_{z\in s(x)\setminus \lbrace y\rbrace}m(z)\leq m(y)+\frac{q(x)-1}{k}m(x).
%\end{equation*}
%It follows
%\begin{equation*}
 %   m(x)\leq \frac{k}{k-q(x)+1}m(y),
%\end{equation*}
%where $c=k/(k-q(x)+1)\geq k/(k-1)>1$.
\end{proof}
In the previous argument it is proved between the lines a fact which is itself important and deserves some attention. We restate it here as a corollary.
\begin{cor}\label{corollario bounded degree} If $T$ admits a locally doubling flow measure, then 
%If a flow measure $m$ is locally doubling, then
\begin{equation*}
    q(x)\leq c, \quad \forall x\in V,
\end{equation*}
with the same constant $c$ as in \eqref{control sons}.
\end{cor}
Observe that the opposite is not true, i.e., not every flow on a tree of bounded degree is locally doubling. In fact, it is clear that any measure $m$ with a super exponential growth along the geodesic $g$, so not satisfying \eqref{control sons}, can be defined outside $g$ in such a way to be a flow.
\begin{oss}\label{remark on c}
Note that unless $T=\mathbb{Z}$, namely the trivial tree where each vertex has exactly two neighborhoods (a predecessor and a son), in fact the constant $c$ in \eqref{control sons} must be greater or equal than 2, as a consequence of Corollary \ref{corollario bounded degree}.
\end{oss}
Remarkably, it turns out that trees with root at infinity do not admit doubling flow measures, unless almost all of their vertices have only one son. Let $n:\Gamma \to \mathbb{N}$ be the function counting the number of vertices having at least two sons along each geodesic,
\begin{equation*}
    n(\gamma)=|\lbrace y\in \gamma: \ q(y)\geq 2\rbrace|.
\end{equation*}
We have the following characterization.

\begin{thm}\label{teorema caratterizzazione alberi con flussi doubling}
A tree $T$ rooted at infinity admits a doubling flow measure if and only if it has uniformly bounded degree and
\begin{equation}\label{sup_geodetiche}
    \sup_{\gamma\in \Gamma} n(\gamma)=M<\infty.
\end{equation}
\end{thm}
{Observe that it is enough to take the supremum in \eqref{sup_geodetiche} over doubly infinite geodesics having one of the endpoints in $\zeta_g$: in fact, if $\zeta, \eta\in \partial T\setminus \lbrace \zeta_g \rbrace$, then clearly $n([\zeta, \eta])\leq n([\zeta, \zeta_g])+n([\eta, \zeta_g])$.} The above theorem is an immediate consequence of the following proposition, which characterizes trees on which the locally doubling condition implies the doubling condition for flows.
\begin{prop}\label{prop: locally doubling is doubling}
A locally doubling flow measure on $T$ is doubling if and only if \eqref{sup_geodetiche} holds.
\end{prop}
\begin{proof}
Let \eqref{sup_geodetiche} hold, $m$ be a locally doubling flow and $c$ be the constant in \eqref{control sons}.  Then for every $x_0\in V, r \ge 2$, it holds
\begin{equation*}
    \frac{m(B_{2r}(x_0))}{m(B_r(x_0))}\leq \frac{(4r+1)m(x_{2r})}{(r-1)m(x_{\ceil{r/2}})}\leq  \frac{(4r+1)c^Mm(x_{\ceil{r/2}})}{(r-1)m(x_{\ceil{r/2}})}\leq 9c^M.
\end{equation*} If $r=1$, we easily get the uniform boundedness of $\frac{m(B_{2}(x_0))}{m(B_1(x_0))}$ by the definition of locally doubling measure. Hence, $m$ is doubling.

Conversely, let $m$ be a doubling flow, $C$ the doubling constant and $x,z\in V$ with $z>x$. Choose $x_0<x$ such that $r=d(x_0,x)=2d(x,z)$. Then,
\begin{equation*}
    C\geq    \frac{m(B_{2r}(x_0))}{m(B_r(x_0))}\geq\frac{2\sum_{j=3r/2}^{2r-1}m(x_j)}{(2r-1)m(x)}\geq\frac{rm(z)}{(2r-1)m(x)}.
\end{equation*}
Hence, $m(z)\leq 3 C m(x)$. On the other hand, $m$ is locally doubling, so by \eqref{lower bound} $m(z)\geq k^{n([p(x),z])} m(x)$. Then, $n([p(x),z])\leq \log_k(3C)$. By the generality of $x$ and $z$ the result follows.
\end{proof}
\begin{proof}[Proof of Theorem \ref{teorema caratterizzazione alberi con flussi doubling}]
Clearly if $T$ admits a doubling measure then it must have bounded degree by Corollary \ref{corollario bounded degree} and satisfies \eqref{sup_geodetiche} by Proposition \ref{prop: locally doubling is doubling}. Conversely, let $T$ be a tree satisfying \eqref{sup_geodetiche} and suppose that $q(x)\leq c$ for every $x\in V$. Then any measure $m$ satisfying $m(p(x))=q(p(x))m(x)$ at every vertex $x$ is a locally doubling flow since $m(p(x))\leq c m(x)$. We conclude by Proposition \ref{prop: locally doubling is doubling}.
\end{proof}

\begin{defn}
We say that $m$ has at least exponential growth if, for all $x_0 \in V$, there exist $r_0 = r_0(x_0) \in \mathbb{N}$, $\beta = \beta(x_0)>0$ and $\alpha = \alpha(x_0) > 0$ such that $m(B_r(x_0)) \geq \beta e^{\alpha r}$ for all $r > r_0$.
\end{defn}

\begin{prop}
Let $m$ be a locally doubling flow. Then $m$ has at least exponential growth if and only if for all $x_0 \in V$ there exist $r_0 = r_0(x_0) \in \mathbb{N}$ and $\alpha = \alpha(x_0) > 0$ such that $n([x_1,x_r])\geq \alpha r$ for all $r > r_0$.
\end{prop}
\begin{proof}
For the sufficient condition, by \eqref{lower bound} we have that for any $x_0\in V$ and $r\geq r_0$
\begin{equation*}
    m(B_r(x_0))\geq m(x_r)\geq m(x_0)k^{n([x_1,x_r])}\geq m(x_0)k^{\alpha r}\geq \beta e^{\alpha r}.
\end{equation*}
Conversely, assume that $m$ has at least exponential growth. Then for some $\alpha, \beta ,r_0$ and all $r>r_0$ we have
\begin{equation*}
    \beta e^{\alpha r}\leq m(B_r(x_0))\leq (2r+1)c^{n([x_1,x_r])}m(x_0),
\end{equation*}
where $c$ is the constant in \eqref{control sons}. Then we get
\begin{equation*}
    c^{n([x_1,x_r])/r}\geq e^{\alpha}\Big(\frac {\beta }{(2r+1)m(x_0)}\Big)^{1/r}\longrightarrow e^{\alpha}>1 \quad \text{as} \ r\to \infty.
\end{equation*}
The assumption that there exists an $x_0 \in V$ such that $\displaystyle\liminf_{r\to\infty}n([x_1,x_r])/r= 0$ would then lead to a contraddiction.
\end{proof}

We end the section by observing that the \textit{isoperimetric inequality} does not hold for flow measures on $T$. We say that a measure $m$ satisfies the isoperimetric inequality on $T$ if there exists a constant $C_{iso}>0$ such that for every bounded $A \subset V$
\begin{equation*}
    m(\partial A)\geq C_{iso} m(A),
\end{equation*}
where the boundary of $A$ is defined as $\partial A=\{x \in A \ : \ \exists \ y \in A^c\  \text{such that} \ y \sim x\}$.
Now, given a ball $B=B_r(x_0)$, $r>0$ and $x_0 \in V$, set $B^+=B\cap\lbrace x\in V: \ x>x_0\rbrace$, $B^-=B\cap\lbrace x\in V: \ x\leq x_0\rbrace$. Note that $B^+\cup B^-\subset B$, with equality holding if and only if $r\leq 1$. \\ \indent A flow measure does not satisfy the isoperimetric inequality. Indeed, % In order to prove the previous claim, which is that the isoperimetric inequality fails for flow measures, it is enough to observe that when $m$ is a flow
\begin{equation*}
    \frac{m(\partial B^-)}{m(B^-)}=\frac{2m(x_0)}{(r+1)m(x_0)}=\frac{2}{r+1}\to 0 \quad \text{as } r\to +\infty.
\end{equation*}

\section{Calder\'on--Zygmund theory}\label{s: admtr}

The classical Calderón-Zygmund theory heavily relies on the fact that metric balls enjoy the doubling property with respect to the given measure. As shown in the previous section, flow measures on the tree tested on balls are typically non-doubling. For this reason, inspired by the seminal work \cite{hs}, we substitute balls with a different family of sets which will show to be doubling in an appropriate sense, and can be used as base sets for building up a Calderón-Zygmund theory in this contest.

\subsection{Admissible trapezoids}

%The classical Calderón-Zygmund theory heavily relies on the fact that metric balls enjoy the doubling property with respect to the given measure. As shown in the previous section, flow measures on the tree tested on balls are typically non-doubling. For this reason, inspired by the seminal work \cite{hs}, we substitute balls with a different family of sets which will show to be doubling in an appropriate sense, and can be used as base sets for building up a Calderón-Zygmund theory in this contest.
%%{\color{purple} dire qualcosa sul fatto che a differenza di HS permettiamo al rapporto tra altezze di variare tra 2 e un certo $\beta$ per poter dividere/espandere in libertà?}
%
%\st{Given a ball $B=B_r(x_0)$, we set $B^+=B\cap\lbrace x\in V: \ x>x_0\rbrace$, $B^-=B\cap\lbrace x\in V: \ x\leq x_0\rbrace$. Observe that $B^+\cup B^-\subset B$, with equality holding if and only if $r\leq 1$.}

For $h''>h'\in \mathbb{N}\setminus\lbrace 0 \rbrace$, we define the \textit{trapezoid} rooted at $x_0\in V$ of parameters $h',h''$ as
\begin{equation*}
    R=R_{h'}^{h''}(x_0)=B_{h''-1}^-(x_0)\setminus B_{h'-1}^-(x_0)=\lbrace x\in V: \ x\leq x_0, \ \ell(x_0)-h''<\ell(x)\leq \ell(x_0)-h'\rbrace.
\end{equation*}
Observe that if $m$ is a flow measure, then
\begin{equation*}
    m(R)=m(x_0)(h''-h'),
\end{equation*}
reason why we call the quantity $h(R)=h''-h'$ the \textit{height} of the trapezoid.

Singletons $\lbrace x_0\rbrace$ are also considered to be trapezoids. Given a number $\beta\geq 12$ \footnote{This specific lower bound on $\beta$ is needed to guarantee enough room to perform the expansion algorithm described below.} we say that a trapezoid $R$ is admissible (with respect to $\beta$) if either $R=\lbrace x_0\rbrace$ or $R=R_{h'}^{h''}(x_0)$ with $2\leq h''/h'\leq \beta$, for some $x_0\in V$. We fix $\beta$ once for all and we denote by $\mathcal{R}$ the corresponding family of admissible trapezoids.

\bigskip 

{\bf{Assumptions.}} From now on we will assume that $T$ is a tree rooted at $\zeta_g\in\partial T$, $V$ is its set of vertices, and 
\begin{itemize}
\item[(1)] $m:V\rightarrow \mathbb R^+$ is a locally doubling flow measure; 
\item[(2)] $T$ is of uniformly bounded degree, i.e., $q(x)\leq c$, $\forall x\in V$. 
\end{itemize}
By Corollary \ref{corollario bounded degree} (2) is a consequence of (1), but we explicitly state it here to recall once for all the notation of the constant $c$.

\subsection{Decomposition and expansion algorithms}

We now describe procedures to define dyadic decompositions and expansions of admissible trapezoids. Let $R=R_{h'}^{h''}(x_0)\in \mathcal{R}$, and set $\gamma=h''/h'$. We have the following \textit{decomposition algorithm}:
\begin{itemize}
    \item if $R=\lbrace x_0 \rbrace$, return $R$;
    \item if $h'=1$ and $h''=2$, cut $R$ in the disjoint union of its vertices: they are at most $c$, all brothers, sons of $x$;
    \item if $h'=1$ and $h''=3$ or $\gamma\geq 4$, cut $R$ horizontally producing 
        \begin{equation*}
            R_u=R_{h'}^{2h'}(x_0), \quad R_d=R_{2h'}^{h''}(x_0);
        \end{equation*}
    \item else, cut $R$ vertically producing,
        \begin{equation*}
            R_y=R_{h'-1}^{h''-1}(y), \quad y\in s(x_0).
        \end{equation*}
\end{itemize}
It is easy to see that in any case the produced sub-trapezoids are admissible. Let $\mathcal{F}(R,1)$ be the output of the algorithm, which is a family of at most $c$ trapezoids forming a partition of $R$, and for $k\geq 1$ let $\mathcal{F}(R,k+1)$ be the family of trapezoids produced by applying the decomposition algorithm to each element of $\mathcal{F}(R,k)$. Observe that the algorithm can be iterated until one reaches the trivial partition of the given trapezoid $R$, which is the one constituted of singletons only.

Conversely, if we want to produce the dyadic father of the given admissible trapezoid $R$, we proceed via the following \textit{expansion algorithm}:
\begin{itemize}
    \item if $R=\lbrace x_0\rbrace$, we expand it to $R'=s(p(x_0))$
    \item if $h'=1$ and $h''=2$, we expand $R$ to $R'=R_1^3(p(x_0))$
    \item if $\gamma\geq 3$, we expand $R$ horizontally to $R'=R_{h'+1}^{h''+1}(p(x_0))$
    \item else, we can decide whether to expand $R$ down vertically to $R'=R_{h'}^{2h''}(x_0)$ or up vertically to $R'=R_{\floor{h'/2}}^{h''}(x_0)$.
\end{itemize}
Observe that no vertical expansion is performed as far as $h'=1$, so that also the up-vertical expansion is always properly defined. It is easy to check that any of the above expansion steps produces a new admissible trapezoid $R'$ which contains $R$. 
The following property can be considered as a substitute for the doubling property in the proposed contest.

\begin{prop}\label{costanti}
Let $R\in \mathcal{R}$, $R'$ its expansion and $Q \in \mathcal{F}(R,1)$. Then, 
\begin{equation*}
    \frac{1}{\widetilde{C}}m (R')\leq m(R)\leq \widetilde{C}m(Q), 
\end{equation*}
where $\widetilde{C}=\max\lbrace 2c,\beta-1, 3\rbrace$.
\end{prop}
\begin{proof}
Let $R=R_{h'}^{h''}(x_0)$ be an admissible trapezoid, $Q\in \mathcal{F}(R,1)$ and $R'$ the dyadic expansion of $R$. The following estimates hold:

\begin{equation*}
\begin{split}
m(R)&\leq
    \begin{cases}
m(Q), \quad &\text{if} \ R=\lbrace x_0\rbrace \\
(\beta-1)m(Q) &\text{if} \ Q=R_u\\
3m(Q)/2, &\text{if} \ Q=R_d\\
cm(Q),  &\text{otherwise},\\
\end{cases} \\ 
 m(R')&\leq
     \begin{cases}
cm(R), \quad &\text{if} \ R=\lbrace x_0\rbrace \ \text{or} \ \gamma\geq 3\\
2cm(R), &\text{if} \ h'=1, h''=2\\
3m(R),  &\text{if} \ R' \ \text{is {down} vertical expansion of} \ R\\
5m(R)/2  &\text{if} \ R' \ \text{is {up} vertical expansion of} \ R. 
\end{cases}
\end{split}
\end{equation*}
\end{proof}
%Observe that, by Remark \ref{remark on c}, $\overline{c}=3$ if $T=\mathbb{Z}$, while $\overline{c}=2c$ in all other cases.
\begin{comment}
 \textcolor{blue}{ 
We remark here that the isoperimetric inequality does not hold for flow measures even when tested on admissible trapezoids: fix a vertex $x_0 \in V$ and define $A_h=R_h^{2h}(x_0).$ We have that $\frac{m(\partial A_h)}{m(A_h)}=\frac{2}{h} \to 0$ as $h \to +\infty$.}
\end{comment}

%Now, let $j(k)$ be the subsequence of indices such that $R_{j(k)}$ is obtained by vertical expansion of $R_{j(k)-1}$ and set $k_0=k_0(j)=\max\lbrace k: \ j(k)\leq j\rbrace$. Define $\mathcal{A}_k^j=\mathcal{B}(R_{j(k+k_0)}\setminus R_{j(k+k_0)-1})$
%Observe that 
%\begin{equation}\label{D_j}
%    V=\mathcal{N}_j\bigcup_{k\geq 1}\mathcal{A}_k^j, \quad j\geq 0.
%\end{equation}
%However, we need to further refine the partition described in \eqref{D_j}. Let $\widetilde{\mathcal{N}}_1^j=\mathcal{A}_1^j$ and for $k\geq 2$, let
%\begin{equation*}
%   \widetilde{\mathcal{N}}_k^j=\lbrace R\in\mathcal{F}(Q,k): \ Q\in\mathcal{A}_k^j\rbrace .
%\end{equation*}
%We define
%\begin{equation*}
%    \mathcal{D}_j=\mathcal{N}_j\bigcup_{k\geq 1}\widetilde{\mathcal{N}}_k^j, \quad j\geq 0,
%\end{equation*}

%and for $\ell\in\mathcal{L}$, $\ell\geq j+1$ let $\mathcal{N}_\ell$ be the family of brothers of $R_\ell\setminus R_{\ell-1}$.
\subsection{Hardy--Littlewood maximal function}
In this section we prove the $L^p$ boundedness of the Hardy--Littlewood maximal function associated to the family $\mathcal{R}$ when $p\in (1,+\infty]$ and  we also show that $M$ is of the weak-type (1,1). %We recall that 
Given a function $f : V \to \mathbb{C}$,  its maximal function  is defined by
\begin{align*}
    Mf(x)= \sup_{R \ni x} \frac{1}{m(R)}\int_ R|f|\  dm,
\end{align*}
where the supremum is taken over all $R \in  \mathcal{R}$ such that $x \in R$. \\ 
Given $R=R_{h'}^{h''}(x) \in  \mathcal{R}$ we define its \textit{envelope} as
\begin{align*}
    \Tilde{R}=R_{\lceil \frac{h'}{\beta}\rceil}^{\lceil\beta h'' \rceil}(x).
\end{align*}
It is easy to show that 
\begin{align*}
    m(\Tilde{R}) \le 2\beta m(R), \quad \forall R \in  \mathcal{R}.
\end{align*}
In order to prove the boundedness of the Hardy--Littlewood maximal function we need  the following technical lemma.
\begin{lem}\label{max1}
Let $R_1, R_2 \in \mathcal{R}$  with roots $x_1,x_2$ respectively, such that $R_1 \cap R_2 \ne \emptyset$ and  $m(x_{1}) \ge m(x_{2})$. Then
\begin{align*}
    R_2 \subset \Tilde{R}_1.
\end{align*}
\end{lem}
\begin{proof} 
If $R_1$ and $R_2$ are singletons, then they coincide. If $R_2=\{x_2\}$, then $R_1 \cap R_2 \ne \emptyset$ implies $x_2 \in R_1 \subset \Tilde{R}_1.$ If $R_1=\{x_1\}$, then $R_1 \cap R_2 \ne \emptyset$ implies $x_1 \in R_2$, but, since $m(x_1) \ge m(x_2)$, it follows $x_1=x_2.$  \\ 
Consider now the case when $R_1=R_{h'_1}^{h''_1}(x_{1}), R_2=R_{h'_2}^{h''_2}(x_{2})$ are both not singletons. Define $\ell_i=\ell(x_i)$. Since
 $R_1 \cap R_2 \ne \emptyset$, there exists $\overline{x} \in R_1 \cap R_2$; hence $\ell(x_i) \ge\ell(\overline{x}). $  It is easy to see that the existence of $\overline{x}$ implies that $x_2$ lies below $x_1$ and in particular $\ell_2 \le\ell_1.$ Moreover,  $\ell_i-h_i''+1 \le \ell(\overline{x}) \le\ell_j-h'_j$, with $i,j=1,2.$ Thus
\begin{align}\label{max}
  \begin{cases}    \ell_1-\ell_2 \ge h'_1-h''_2+1, \\ 
   \ell_1-\ell_2 \le h''_1-h'_2 -1.
    \end{cases}
\end{align}
Let $x$ be a vertex in $R_2.$ By definition we have $h'_2 \le\ell_2-\ell(x) \le h''_2-1.$ By \eqref{max} 
\begin{align*}
    \ell_1-\ell(x)&=\ell_1-\ell_2+\ell_2-\ell(x) 
    %&\le \beta(\ell_1-\ell_2)+\ell_2-\ell(x) \\ 
    %&\le \beta(h''_1-h'_2+1)+h_2''-1 \\ 
    <\beta(h''_1-\frac{h''_2}{\beta} +1)+h''_2-1 
    <\beta h_1'' \le \lceil \beta h_1'' \rceil. 
\end{align*}
Again, by \eqref{max},
\begin{align*}
     \ell_1-\ell(x) &= \ell_1-\ell_2+\ell_2-\ell(x) 
    % &\ge \frac{1}{\beta}(\ell_1- \ell_2)+ \ell_2 - \ell(x) \\ 
     %&\ge \frac{1}{\beta}(h'_1-h''_2+1)+ h'_2 \\ 
     > \frac{1}{\beta}(h'_1-\beta h'_2+1)+h'_2  
     >\frac{h'_1}{\beta},
\end{align*}
hence, we deduce $\ell_1-\ell(x) \ge \lceil \frac{h'_1}{\beta} \rceil.$
In conclusion, ${x} \in \Tilde{R}_1.$
\end{proof}

We remark that  \begin{align}\label{stella}
    \|Mf\|_\infty \le \|f\|_\infty, \qquad{\forall f \in L^\infty(m).}
\end{align} We can now state the main result of this section.
\begin{thm}
The following hold.
\begin{itemize}
    \item[(i)] For all $f \in L^1(m)$ and $\lambda>0$ $$m(\{x \in V \ : \ Mf(x)>\lambda\}) \le \frac{2\beta}{\lambda}\|f\|_1;$$
    \item[(ii)] for every $p \in (1,+\infty)$, $M$ is bounded on $L^p(m)$  with constant at most $2\bigg(2\beta\frac{p}{p-1}\bigg)^{1/p}.$
\end{itemize}
\end{thm}
\begin{proof} Property
$(ii)$  follows by $(i)$ and  \eqref{stella} by  the Marcinkiewicz interpolation Theorem. For proving $(i)$, by means of Lemma \ref{max1}, we can follow closely the proof of \cite[Th. 3.1]{hs}. \\ 
Let $\lambda>0$, $f \in L^1(m)$ and set 
\begin{align*}
    \Omega_\lambda&=\{x \in V \ : \ Mf(x)>\lambda \}, \quad S_0=\bigg\lbrace R \in  \mathcal{R} \ : \ \frac{1}{m(R)}\int_{ R}|f|\ dm>\lambda \bigg\rbrace.
\end{align*}
For all $R \in S_0$ we have 
$$
 m(x_R)\leq    m(R) < \frac{1}{\lambda} \|f\|_1 .
$$
$S_0$ is countable hence we can introduce an order. We say that $R \ge Q$ if $m(x_R) \ge m(x_Q).$ Let $R_1$ be the maximal trapezoid in $S_0$ with respect to $\ge$ (it exists because of the previous estimate) which appears first in the order. Define $S_1=\{R \in S_0 \ : \ R \cap R_1 = \emptyset\}.$ \\  Let $R_2$ be the  maximal admissible trapezoid in $S_1$ which appears first in the ordering. So we can define inductively the sequences \begin{align*}
    S_{i+1}=\{R \in S_i \ : \ R \cap R_j= \emptyset, \ j \le i \}, 
\end{align*} and $R_{i+1}\in S_{i+1}$ {is the maximal trapezoid with respect to $\ge$ which appears first in the ordering.}  We claim that 
\begin{align}\label{claim}
    \forall R \in S_0 \  \exists R_i \ : \ R \cap R_i \ne \emptyset,\  m(x_{R_i}) \ge m(x_R).
\end{align}
By Lemma \ref{max1}, \eqref{claim} in particular implies that $R \subset \Tilde{R}_i$. Now we prove the claim: it suffices to show that there exists $j \in \mathbb{N}$ such that $R \in S_j\setminus S_{j+1}$. By contradiction, if such a $j$ does not exist, then $\exists \ k$ such that $S_k$ contains infinite admissible trapezoids $\{T_l\}_l$ such that $T_l \cap T_i = \emptyset$ if $i \ne j$, $m(x_{T_l})=\max\{m(x_R) : R \in S_k\}$, $T_l \cap R = \emptyset.$ Now we set 
\begin{align*}
    R_k={T_{1}},..., R_{k+i}=T_{i+1}, \dots,
\end{align*}
then
\begin{align*}
    \sum_{i=k}^{+\infty}m(x_{R_i})\le \sum_{i=k}^{+\infty}m(R_i) \le \sum_{i=k}^{\infty}\frac{1}{\lambda}\int_{ R_i}|f|\ dm \le \frac{\|f\|_1}{\lambda},
\end{align*}
and the left hand side is infinite. Thus the claim is proved. 
Define $E=\cup_i \Tilde{R}_i$ and notice that $
    E^c \subset \Omega_\lambda^c$. Indeed if $x \in E^c$ and $R \in  \mathcal{R}$ contains $x$ then $R \notin S_0$. We conclude that  for $x \in E^c$
\begin{align*}
    M(f)(x)= \sup_{R \ni x, R \notin S_0 } \frac{1}{m(R)}\int_{R}|f|\  dm \le \lambda,
\end{align*}
hence $x\notin \Omega_\lambda$. In conclusion
\begin{equation*}
 m(\Omega_\lambda)\leq   m(E) \le \sum_i m(\Tilde{R}_i) \le 2\beta \sum_{i}m(R_i) \le 2\beta \frac{1}{\lambda}\|f\|_1. \qedhere
\end{equation*} 
\end{proof}

\subsection{Calder\'on--Zygmund decomposition}\label{subs: CZ}
 The aim of this section is to introduce a Calder\'on--Zygmund decomposition in our setting. We initially prove a preliminary lemma which states the existence of a partition of $V$ consisting of big admissible trapezoids, in the sense that, if we fix any $\sigma>0$, each set of such a partition has measure larger than $\sigma$.
     \begin{lem}\label{10}
     For all $\sigma>0$ there exists a partition $\mathcal{P} \subset \mathcal{R}$ of $V$, such that 
\begin{align*}
    m(R)> \sigma, \qquad \forall R \in \mathcal{P}.
\end{align*}
\end{lem}
\begin{proof}
For all $n\in \mathbb{Z}$ let $x_n$ denote the vertex in $g$ such that $\ell(x_n)=n$. We consider two cases, either $\{m(x_n)\}_{n \in \mathbb{Z}}$ diverges as $n \to + \infty$ or $\{m(x_n)\}_{n \in \mathbb{Z}}$ is bounded. If $m(x_n) \to +\infty$, then there exists $\overline{n} \in \mathbb{N}$ such that $m(x_n)>{\sigma}{c}$ for all $n \ge \overline{n}$ where $c$ is as in \eqref{control sons}. For any $y\in V$ set $R_l^y=R_{2^{l-1}}^{2^{l}}(y)$ for all $l \in \mathbb{N}$. We define $\mathcal{P}$ as
\begin{align*}
    \mathcal{P}= \lbrace R_l^y: \ l\in \mathbb{N}, \  y \in \lbrace x_{\overline{n}-1}\rbrace \cup (s(x_n)\setminus g), \ n\geq \overline{n} \rbrace\cup \lbrace s(x_j): j\geq \overline{n} \rbrace.
\end{align*}
This concludes the proof when $\{m(x_n)\}_{n \in \mathbb{Z}}$ diverges. 

\begin{comment}
{\color{red} 

For all $n \ge \overline{n}$ such that $q(x_n)\ge 2,$  let $x^{n_i}$ denote the $i$-th son of $x_n$ in $s(x_n) \setminus g$ and  set  $R_l^{n_i}=R_{2^{l-1}}^{2^{l}}(x^{n_i})$ for all $l \in \mathbb{N}$. We define $\mathcal{P}$ as
\begin{align*}
    \mathcal{P}= \bigg(\bigcup_{l, n_i}R_l^{n_i}\bigg) \bigcup  \bigg(\bigcup_{j \ge \overline{n}} s(x_j)\bigg)  \bigcup     \bigg(\bigcup_{m \ge 1}R_{2^{m-1}}^{2^m}(x_{\overline{n}-1})\bigg),
\end{align*}
%where $n \ge \overline{n}$, such that $q(x_n) \ge 2$, $l \in \mathbb{N}$ and $i=i(n)=1,.., |s(x_n)|-1.$ \\  
where the first union is taken over all $l \ge 1$ and $n_i$ such that $n \ge \overline{n}$, $q(x_n) \ge 2$ and $i=1,...,q(x_n)-1$.

}
\end{comment}

Now assume that $\{m(x_n)\}_{n \in \mathbb{Z}}$ is bounded. By \eqref{lower bound}, there are finitely many vertices $y \in g$ such that $\ell(y) \ge 0$ and $q(y) \ge 2.$ Let  $x_l$ denote the vertex in $g$ with maximum level such that $q(x_l) \ge 2.$ By the  definition of flow we have that $m(x_n)=m(x_{l})$ if $n \ge l$. First, notice that there exists $p \in \mathbb{N}$ such that $2^{p-1} m(x_l)> \sigma$, thus we can cover the upper part of the tree with trapezoids $U_k=R_{2^{p-1}}^{2^p}(x_{l+2^{p-1}k})$ where $k \ge {1}$ and $m(U_k)=2^{p-1}m(x_l)>\sigma$ for all $k$. Subsequently, we cover the lower part of the tree with trapezoids $L_j=R_{2^j}^{2^{j+1}}(x_{l+2^{p-1}})$ with $j \ge p.$ Observe that $m(L_j)=2^jm(x_l) \ge 2^p m(x_l)> \sigma.$ We conclude by defining 
\begin{equation*}
    \mathcal{P}= \{U_k\}_{k \ge 1}  \bigcup \{L_j\}_{j \ge p}. \qedhere
\end{equation*}
\end{proof}

Next lemma provides a quite general procedure to determine a family of stopping sets for a given testing condition on the size of the $L^1$ mean of a function. Several  results in the paper will rely on such a scheme as a basic step.
\begin{lem}\label{stopping lemma}
Let $f:V\to\mathbb{C}$, $\alpha>0$ and $R \in \mathcal{R}$  such that $\frac{1}{m(R)}\int_R |f|  \ dm<\alpha$. Then, there exists a family $\mathcal{F}$ of disjoint admissible trapezoids such that for each $E\in \mathcal{F}$ the following hold:
\begin{itemize}
    \item[(i)] $\displaystyle\frac{1}{m(E)}\int_{E} |f| \  dm \ge \alpha$;
    \item [(ii)] $\displaystyle\frac{1}{m(E)}\int_{E} |f|\   dm  <\widetilde{C}\alpha$;
    \item[(iii)] if $ {x \in  R \setminus \bigcup_{E\in \mathcal{F}} E}$, then $|f(x)| < \alpha$.  
    \end{itemize}
\end{lem}
\begin{proof}
 We apply the decomposition algorithm to $R$: if $Q\in \mathcal{F}(R,1)$ is such that $\frac{1}{m(Q)}\int_{Q} |f| \ dm \ge \alpha$ then we stop and declare $Q\in \mathcal{F}$, otherwise, if $Q$ is divisible (i.e., it is not a singleton) we split it up applying again the decomposition algorithm.
 We iterate the above reasoning until $R$ is partitioned in some \textit{stopping sets} $E$ such that $\frac{1}{m(E)}\int_{E} |f| \ dm \ge \alpha$ (some of which may be singletons) and some singletons $x$ at which $|f(x)|=\frac{1}{m(x)}\int_{\lbrace x\rbrace} |f|  \ dm<\alpha$ . Let $\mathcal{F}$ be the family of the stopping sets. Then $(i)$ and $(iii)$ hold by construction. To prove $(ii)$:  for each $E\in \mathcal{F}$ there exists $k\geq 1$ such that $E\in \mathcal{F}(R,k)$. Let $E'$ be the unique set in $ \mathcal{F}(R,k-1)$ such that $E \in \mathcal{F}(E',1)$.
 Then $E\subset E'$, $m(E') \le \widetilde{C} m(E)$ and, since $E'$ is not a stopping set, $\frac{1}{m(E')}\int_{E'} |f| \  dm < \alpha$. Hence \begin{equation*}
        \frac{1}{m(E)}\int_{E} |f|\   dm \le \frac{\widetilde{C}}{m(E')} \int_{ E'} |f| \ dm <\widetilde{C}\alpha.\qedhere
    \end{equation*}
\end{proof}
Now we present the main result of this section, namely the Calder\'on--Zygmund decomposition of integrable functions. 

\begin{thm}[Calder\'on--Zygmund decomposition] For every $f\in L^1(m)$ and $\alpha>0$, there exist a family $\lbrace E_i\rbrace$ of disjoint admissible trapezoids and functions $g$, $b^i$ such that $f=g+\sum_i b^i$ and

\begin{itemize}
    \item[(i)] $|g| \le \widetilde{C} \alpha$;
    \item[(ii)]  $b^i=0$ on $(E_i)^c$;
    \item[(iii)]  $\|b^i\|_1 \le 2\widetilde{C}\alpha m(E_i)$ and $\int_{E_i} \ b^i \ dm = 0;$
    \item[(iv)] $\displaystyle\sum_i m(E_i) \le \frac{\|f\|_1}{\alpha}$.
\end{itemize}
\end{thm}
\begin{proof}
Let $\mathcal{P} \subset \mathcal{R}$ be a partition of ${V}$  such that for all  $R \in \mathcal{P}$ we have  $m(R)>\frac{\|f\|_1}{\alpha}$. Then, for every $R \in \mathcal{P}$, it holds $\frac{1}{m(R)}\int_R |f|  \ dm<\alpha$, and  we can apply Lemma \ref{stopping lemma}. Let $\mathcal{F}(R)$ be the family of stopping sets generated by $R$ and let  $\lbrace E_i \rbrace$ be a listing of the sets belonging to $\mathcal{F}(R)$ for some $R \in \mathcal{R}$.
 We define now 
\begin{align*}
    &g(x) = \begin{cases}  \displaystyle\frac{1}{m(E_i)} \int_{E_i} f \ dm   &\text{if $x \in\bigcup_i E_i$,} \\ 
    f(x) &\text{else,}\end{cases} \\ 
    &b^i(x) = \bigg(f(x)- \frac{1}{m(E_i)}\int_{ E_i} f \ dm \bigg) \chi_{E_i}(x).  
\end{align*}
By  Lemma \ref{stopping lemma}, $|g(x)| \le \widetilde{C} \alpha.$ Every $b^i$ is supported  in $E_i$ and $\int_{E_i}{b^i} \ dm=0$. Moreover,
\begin{align*}
    \|b^i\|_1 \le 2 \int_{E_i} |f|\ dm \le 2\widetilde{C}\alpha m(E_i),
\end{align*}
and
\begin{equation*}
    \sum_i m(E_i) \le \frac{1}{\alpha} \sum_i \int_{ E_i} |f|\  dm \le \frac{\|f\|_1}{\alpha}.\qedhere
\end{equation*}
\end{proof}

  \section{${BMO}$ and Hardy spaces}\label{s: BMOHardy}
  This section is devoted to the definitions and the study of properties of $BMO$ and Hardy spaces in our setting.
  \subsection{$BMO$ spaces}
    We now introduce the space of bounded mean oscillation functions.
  In the following we denote by $f_R$ the integral average of $f$ on $R$, i.e., 
  \begin{align*}
      f_R= \frac{1}{m(R)}\int_{R}f \ dm.
  \end{align*} 

\begin{defn}
Given $q \in [1,+\infty)$ we define $BMO_q(m)$ as the space of all functions $f:V\to \mathbb{C}$ such that 
   \begin{align*}
        \|f\|_{BMO_q}=\sup_{R \in  \mathcal{R}} \bigg(\frac{1}{m(R)} \int_{R}|f-f_R|^q \ dm\bigg)^{1/q}<\infty,
   \end{align*}
   quotiented over constant functions. It can be easily shown that $(BMO_q(m), \|\cdot\|_{BMO_q})$ is a Banach space. 
  \end{defn}

\begin{comment}
 \begin{defn}
Given $q \in [1,+\infty)$   we define $\mathcal{BMO}_q(m)$ as the space of all functions $f :V\to \mathbb{C}$ such that 
   \begin{align*}
       \sup_{R \in  \mathcal{R}} \bigg(\frac{1}{m(R)} \int_{R}|f-f_R|^q \ dm\bigg)^{1/q}.
   \end{align*}
   The space $BMO_q(m)$ is the quotient of $\mathcal{BMO}_q(m)$ with constant functions. We endow $BMO_q(m)$ with the norm 
   \begin{align*}
       \|f\|_{BMO_q}=\sup_{R \in  \mathcal{R}}\bigg( \frac{1}{m(R)} \int_{R}|f-f_R|^q \ dm\bigg)^{1/q}.
   \end{align*}
   It can be easily shown that $(BMO_q(m), \|\cdot\|_{BMO_q})$ is a Banach space. 
  \end{defn}
\end{comment}

As an immediate consequence of the H\"older's inequality we have 
  \begin{align*}
      \|f\|_{BMO_1} \le \|f\|_{BMO_q},
  \end{align*}
  thus $BMO_{q}(m) \subset BMO_1(m)$. To prove the reverse inclusion we shall first show that a  John-Nirenberg inequality holds in our setting. 
  
  \begin{thm}[John-Nirenberg inequality]\label{JN} There exist $\eta,A>0$ such that, for all $f \in BMO_1(m)$:
   \begin{itemize}
    \item[(i)] $\displaystyle\sup_{R \in \mathcal{R}_2^\beta} \frac{1}{m(R)} \int_R \exp{\bigg(\frac{\eta}{\|f\|_{BMO_1}}|f-f_R|\bigg)} \ dm \le A$;
         \item[(ii)]  $\displaystyle m(\{x \in R \ : \ |f(x)-f_R|>t\|f\|_{BMO_1}\}) \le A e^{-\eta t}m(R)$,  \qquad $\forall t>0$ and $R \in \mathcal{R}$.
         \end{itemize} 
  \end{thm}
  \begin{proof} Suppose  $f:V\to\mathbb{C}$ is non constant, otherwise the result is trivial. Let $R_0 \in \mathcal{R}$.   If $R_0=\{x_0\}$, then  $f_{R_0}=f(x_0)$ and 
  \begin{align*}
       \frac{1}{m(R_0)}\int_{R_0}\exp{\bigg(\frac{\eta}{\|f\|_{BMO_1}}|f-f_{R_0}|\bigg)}\ dm=1,
  \end{align*}
  thus it is sufficient to choose $A \ge 1$. \\ 
 If $R_0 \ne \{x_0\}$, we have
  \begin{align*}
      \frac{1}{m(R_0)}\int_{R_0}|f-f_{R_0}|\ dm < 2 \|f\|_{BMO_1}.  
  \end{align*}
  
 Applying Lemma \ref{stopping lemma} to the function $f-f_{R_0}$ with $\alpha=2 \|f\|_{BMO_1}$, we get a family $\mathcal{F}$ of disjoint stopping sets contained in $R_0$ satisfying properties $(i)$, $(ii)$ and $(iii)$ in the lemma.
 In particular by $(i)$ follows
\begin{equation}\label{extra property}
\begin{split}
          m\big(\bigcup_{E\in \mathcal{F}} E\big) &=\sum_{E\in \mathcal{F}}m(E)< \frac{1}{2\|f\|_{BMO_1}} \sum_{E\in \mathcal{F}} \int_{E}|f-f_{R_0}|\ dm \\ 
      &\le \frac{1}{2\|f\|_{BMO_1}}\int_{R_0} |f-f_{R_0}| \ dm  \le \frac{m(R_0)}{2}.
      \end{split}
\end{equation}
For each stopping set $E\in \mathcal{F}$ we have
  \begin{equation}\label{sotto a 4}
  \begin{split}
      |f_{E}-f_{R_0}| &\le |f_{E}-f_{E'}|+|f_{E'}-f_{R_0}|\le \frac{1}{m(E)} \int_{E}|f-f_{E'}|\ dm+\frac{1}{m(E')}\int_{E'} |f-f_{R_0}|\  dm \\ 
      &\le \frac{\widetilde{C}}{m(E')}\int_{E'}|f-f_{E'}| \ dm+2\|f\|_{BMO_1} \le(\widetilde{C}+2)\|f\|_{BMO_1}.
      \end{split}
\end{equation}
Now, suppose first $f \in L^\infty(m)$, and for $t>0$ we define 
  \begin{align*}
      F(t)= \sup_{R \in \mathcal{R}_2^\beta} \frac{1}{m(R)} \int_{R}\exp{\bigg(\frac{t}{\|f\|_{BMO_1}}|f-f_R|\bigg)}\ dm.
  \end{align*} 
 Then $|f-f_R| \le 2 \|f\|_\infty$, from which follows
  \begin{align*}
      F(t) \le \exp \bigg( \frac{2t\|f\|_\infty}{\|f\|_{BMO_1}} \bigg)<+ \infty, \quad \forall t>0.
  \end{align*}  
We estimate
 \begin{equation}\label{LastExp}
 \begin{split}
    &\frac{1}{m(R_0)} \int_{R_0} \exp{ \bigg( \frac{t}{\|f\|_{BMO_1}}|f- f_{R_0}| \bigg)}\  dm  \le \frac{1}{m(R_0)} \int_{R_0 \setminus \cup_{E\in \mathcal{F}} E} \exp{(2t)}\ dm  \\ &+\frac{1}{m(R_0)} \sum_{E\in \mathcal{F}} \int_{E} \exp{ \bigg(\frac{t}{\|f\|_{BMO_1}}|f-f_{E}|}\bigg)\exp{\bigg( \frac{t}{\|f\|_{BMO_1}}|f_{E}-f_{R_0}| \bigg)\  dm}.
\end{split}
\end{equation}
 Using \eqref{sotto a 4} we dominate the last expression in \eqref{LastExp} with 
 \begin{align*}
    \exp(2t)+&\frac{1}{m(R_0)}\sum_{E\in \mathcal{F}}\int_E \exp\bigg((\widetilde{C}+2)t\bigg)\exp{\bigg( \frac{t}{\|f\|_{BMO_1}}|f-f_E|\bigg)}\ dm \\
     &\le \exp(2t) +\exp\bigg((\widetilde{C}+2)t\bigg)\frac{1}{m(R_0)} \sum_{E\in \mathcal{F}}m(E)F(t) \\ &\le \exp(2t) +\exp\bigg((\widetilde{C}+2)t\bigg) \frac{F(t)}{2},
 \end{align*}
where the last inequality is verified by \eqref{extra property}. We conclude that $\displaystyle F(t) \le \frac{2 e^{2t}}{2-e^{(\widetilde{C}+2)t}},$
hence there exist $\eta, A>0$ such that $F(\eta)\leq A$. This ends the proof  when $f$ is a bounded function.  \\ 
For the general case, let $f \in BMO_1(m)$ and for all $k \in \mathbb{N}$ and $x \in V$ define
\begin{align*}
    f_k(x)=\begin{cases} f(x) &\qquad{|f(x)| \le k}, \\ 
   \displaystyle \frac{f(x)}{|f(x)|}k &\qquad{|f(x)|>k.}
    \end{cases}
\end{align*}
It is readily seen that $f_k \in L^\infty(m)$, $f_k \to f$ pointwise on $V$, $(f_k)_R \to f_R$ and there exists a positive constant $c'$ such that $\|f_k\|_{BMO_1} \le c'\|f\|_{BMO_1}$. We have that 
\begin{equation*}
\frac{1}{m(R)} \int_{R} \exp{ \bigg( \frac{\eta}{c'\|f\|_{BMO_1}}|f_k-(f_k)_R| \bigg)}\ dm \le \frac{1}{m(R)}\int_{R} \exp{ \bigg( \frac{\eta}{\|f_k\|_{BMO_1}}|f_k-(f_k)_R| \bigg)}\ dm \le C. 
\end{equation*}
Passing to the limit, we deduce {\it (i)} by the dominated convergence theorem. In order to prove {\it (ii)}, notice that
\begin{align*}
    & m(\{x \in R \ : \ |f(x)-f_R|>t\|f\|_{BMO_1}\})=m(\{x \in R : \exp{\bigg(\frac{\eta}{\|f\|_{BMO_1}}|f(x)-f_R|\bigg)}>e^{\eta t}\}) \\ &\le e^{-\eta t}\int_{ R}\exp{\bigg(\frac{\eta}{\|f\|_{BMO_1}}|f-f_R|\bigg)}\ dm\le A e^{-\eta t}m(R), 
\end{align*} 
where the last inequality follows by $(i)$.
  \end{proof} 
  
  A remarkable consequence of Theorem \ref{JN} is the equivalence of the $BMO_q(m)$ spaces, $q \in [1,+\infty)$. 
  \begin{cor}
  For all $1<q<+\infty$ there exists a constant $B_q$ depending only on $q$ such that 
  \begin{align*}
      \|f\|_{BMO_q} \le B_q \|f\|_{BMO_1}, \quad \forall f\in BMO_1(m).
      \end{align*}
  \end{cor}
  \begin{proof}
 \begin{align*}
     \frac{1}{m(R)} \int_{R} |f-f_R|^q \ dm &= \frac{q}{m(R)}\int_{0}^{\infty} \alpha^{q-1}m(\{x \in R : |f-f_R|(x)>\alpha\}) \ d\alpha  \\ 
     &\le q\int_0^\infty \alpha^{q-1} Ae^{-\eta \alpha/\|f\|_{BMO_1}} \ d\alpha \le qA \bigg(\frac{\|f\|_{BMO_1}}{\eta}\bigg)^q\Gamma(q).
 \end{align*}
 We conclude by choosing $B_q={(qA\Gamma(q))^{1/q}}/{\eta}.$
  \end{proof}
 Henceforward,  let $BMO(m)$ denote the space $BMO_1(m)$.
  \subsection{Hardy spaces}
  In this section we introduce atomic Hardy spaces. In our setting, atoms are supported in admissible trapezoids. 
  \begin{defn}
  A function $a$ is a  $(1,p)$-atom for $p \in (1,+\infty]$ if the following hold 
  \begin{itemize}
      \item[$(i)$] $a$ is supported in a set $R\in \mathcal{R}$;  
      \item[$(ii)$] $\|a\|_p \le m(R)^{1/p-1}$;  
      \item[$(iii)$] $ \int_{R} a \ dm=0$.
  \end{itemize}
  \end{defn}
  \begin{defn}
  The Hardy space $H^{1,p}(m)$ is the space of all the function $g \in L^1(m)$ such that $g=\sum_{j} \lambda_ja_j$ where $a_j$ are $(1,p)$ atoms and $\lambda_j$ are complex numbers such that $\sum_j |\lambda_j|<+\infty$. We denote by $\|g\|_{H^{1,p}}$ the infimum of $\sum_j|\lambda_j|$ over all the possible decompositions  $g=\sum_j \lambda_ja_j$ with $a_j$ $(1,p)$-atoms.
  \end{defn}
  We also introduce the subspace \begin{align*}
      H^{1,p}_{\text{fin}}(m)=\bigg\{g \in H^{1,p}(m) \ : \ g=\sum_{j=1}^N \lambda_j a_j, \ N \in \mathbb{N}\bigg\}.
  \end{align*}
  % \subsection{Equivalence of $H^{1,p}(m)$ spaces}
  The next result yields the equivalence of the $H^{1,p}(m)$ spaces when $p\in(1,+\infty].$ It is readily seen that $H^{1,\infty}(m) \subset H^{1,p}(m)$. For the converse, we use a variant of the Calder\'on--Zygmund decomposition, as follows.
  \begin{prop}\label{EqH} For any $p\in (1,+\infty)$ there exists $A_p>0$ such that the following estimate holds $$\|f\|_{H^{1,\infty}} \le A_p \|f\|_{H^{1,p}}, \qquad{\forall f \in H^{1,p}(m)}.$$ %with $A_p$ as in Proposition \ref{EqH}.
  Hence $ H^{1,p}(m)=H^{1,\infty}(m)$ and the norms   $\|\cdot\|_{H^{1,\infty}}$ and   $\| \cdot \|_{H^{1,p}}$ are equivalent. %There exists a constant $A_p$ depending only on $p \in (1,+\infty)$ such that for every $(1,p)$-atom $a$,  one has
 % \begin{align*}
  %    \|a\|_{H^{1,\infty}} \le A_p.
%  \end{align*}
  \end{prop}\begin{proof} It suffices to prove that there exists a constant $A_p$ depending only on $p \in (1,+\infty)$ such that, for every $(1,p)$-atom $a$,  one has
  \begin{align}\label{eqH2}
      \|a\|_{H^{1,\infty}} \le A_p.
  \end{align}
%Before we proceed with the proof, we remark that Proposition \ref{EqH} yields the following 
 % \begin{cor}
%For any $p\in (1,+\infty)$ the following estimate holds $$\|f\|_{H^{1,\infty}} \le A_p \|f\|_{H^{1,p}}, \qquad{\forall f \in H^{1,p}(m)},$$ with $A_p$ as in Proposition \ref{EqH}.
 % Hence $ H^{1,p}(m)=H^{1,\infty}(m)$ and the norms   $\|\cdot\|_{H^{1,\infty}}$ and   $\| \cdot \|_{H^{1,p}}$ are equivalent.\end{cor} 
  %[Proof of Proposition \ref{EqH}]
  Let $a$ be a $(1,p)$-atom. We have that supp$(a) \subset Q \in  \mathcal{R}$, $\|a\|_p \le m(Q)^{1/p-1},$ $\int_{Q} a\ dm=0$. We define $b=m(Q)a$; we claim that $\forall n \in \mathbb{N}$, we can write 
  \begin{align*}
      b=\sum_{l=0}^{n-1}\widetilde{C}^{1/p}\alpha^{l+1}\sum_{j_l \in \mathbb{N}^l}m(R_{j_l})a_{j_l}+ \sum_{j_n \in \mathbb{N}^n}f_{j_n},
  \end{align*}
    where $\alpha>0$ is a constant to be chosen later on, $\widetilde{C}$ is as in Proposition \ref{costanti} and

  \begin{itemize}
      \item[$(i)$] $a_{j_l}$ is a $(1,\infty)$-atom supported in $R_{j_l}$, 
      \item[$(ii)$] supp $f_{j_n}\subset R_{j_n},$ $\int_{R_{j_n}} f_{j_n}\ dm=0,$ 
      \item[$(iii)$] $\bigg(\frac{1}{m(R_{j_n})}\int_{ R_{j_n}} |f_{j_n}|^p\  dm \bigg)^{1/p} \le  {2} \widetilde{C}^{1/p}\alpha^n,
      $ 
      \item[$(iv)$] $\sum_{j_n} \|f_{j_n}\|_p^p \le 2^{pn}\|b\|_p^p$, 
      \item[$(v)$] $|f_{j_n}(x)| \le b(x)+\widetilde{C}^{1/p}\alpha^n2^{n-1}$, 
      \item[$(vi)$] $\sum_{j_n} m(R_{j_n}) \le 2^{p(n-1)}\alpha^{-np}\|b\|_p^p.$
  \end{itemize}
  Assume the claim holds. Then
  \begin{align*}
      \Big\|\sum_{j_n \in \mathbb{N}^n} f_{j_n}\Big\|_1 &\le \sum m(R_{j_n})^{1/p'}\|f_{j_n}\|_{p} \le  {2}\sum m(R_{j_n})^{1/p'}m(R_{j_n})^{1/p}\widetilde{C}^{1/p}\alpha^n \\ 
      &\le  {2} \widetilde{C}^{1/p}\alpha^n 2^{p(n-1)}\alpha^{-np}\|b\|_p^p 
      \le  {2} \widetilde{C}^{1/p}2^{-p}(\alpha^{-p}2^p)^n m(Q),
  \end{align*}
  the last quantity tends to zero as $n \to +\infty$ if $\alpha>2^{\frac{p}{p-1}}.$ The previous computation shows that
  \begin{align*}
      b=\sum_{l=0}^{\infty} \widetilde{C}^{1/p}\alpha^{l+1}\sum_{j_l \in \mathbb{N}^l} m(R_{j_l})a_{j_l}
  \end{align*}
  where the series converges in $L^1(m)$. By properties $(vi)$ we have 
  \begin{align*}
      \sum_{l=0}^{\infty} \widetilde{C}^{1/p}\alpha^{l+1}\sum_{j_l \in \mathbb{N}^l} m(R_{j_l}) &\le \sum_{l=0}^{\infty} \widetilde{C}^{1/p}\alpha^{l+1}2^{p(l-1)}\alpha^{-lp}m(Q)=A_pm(Q),
  \end{align*}
  if $ \alpha>2^{p/(p-1)}$ and we conclude that $\|a\|_{H^{1,\infty}} \le A_p.$ \\ We now prove the claim by induction. Fix $n=1$ and notice that
  \begin{align*}
  \frac{1}{m(Q)} \int_{Q}|b|^p \ dm &= \frac{1}{m(Q)}m(Q)^{p} \int_Q{|a|^p} \ dm  \le 1<\alpha^p. 
  \end{align*}
  Apply Lemma \ref{stopping lemma} to the function $|b|^p$ with the constant $\alpha^p$, call $\{R_i\}_i$ the family of stopping sets and set $E=\cup_i R_i$. Define 
  \begin{align*}
      b=g+\sum_i f_i, \quad f_i=\bigg[b-\frac{1}{m(R_i)}\int_{R_i}b \ dm\bigg]\chi_{R_i}.
  \end{align*}
  By definition of $R_i$, $|g| < \alpha$ on $E^c$, and by  the H\"older's inequality and Lemma \ref{stopping lemma}, we have \begin{align}\label{HS}
      \bigg|\frac{1}{m(R_i)} \int_{R_i} b \ dm\bigg| <  \widetilde{C}^{1/p}\alpha,
  \end{align}
  which yields 
\begin{align*}
    \|f_i\|_p < \bigg(\int_{R_i} |b|^p \ dm \bigg)^{1/p}+ \widetilde{C}^{1/p}\alpha m(R_i)^{1/p} < 2 \widetilde{C}^{1/p}\alpha m(R_i)^{1/p}. 
\end{align*}
Moreover, by \eqref{HS} 
\begin{align*}
    |g(x)| \le \widetilde{C}^{1/p} \alpha \qquad{ \text{if} \ x \in R_i},
\end{align*}
thus $a_0= (\widetilde{C}^{1/p}\alpha m(Q))^{-1}g$ is a $(1,\infty)$-atom. We can write $b=g+\sum_i f_i = \widetilde{C}^{1/p} \alpha m(Q)a_0+\sum_i f_i$, obviously supp$f_i \subset R_i$, $\int f_i \ dm=0$ and   {$\|f_i\|_p \le 2\widetilde{C}^{1/p}\alpha m(R_i)^{1/p}$}. By definition of stopping set and $f_i$ we have 
\begin{align*}
    \sum_i m(R_i) \le \frac{\|b\|_p^p}{\alpha^p}, \quad \|f_i\|_p \le 2\|b\|_{L^p(R_i)},
\end{align*}
hence 
\begin{align*}
    \sum_i \|f_i\|_p^p \le 2^p \|b\|_p^p,
\end{align*}
and the claim is verified. \\ 
We now assume that the claim holds for $n \in \mathbb{N}.$ Then, for all $j_n \in \mathbb{N}^n$,
\begin{align*}
    \frac{1}{m(R_{j_n})} \int_{R_{j_n}} |f_{j_n}|^p \ dm \le  {2^p}\widetilde{C}\alpha^{np}<\alpha^{(n+1)p},
\end{align*} if we choose $\alpha> {2}\widetilde{C}^{1/p}.$ We apply Lemma \ref{stopping lemma} to each $R_{j_n}$ producing stopping sets $R_{j_n i}$, $i \in \mathbb{N}$, such that  
\begin{align*}
    \alpha^{(n+1)p} \le \frac{1}{m(R_{j_n i})} \int_{R_{j_n i}}|f_{j_n}|^p \ dm < \widetilde{C} \alpha^{(n+1)p}.
\end{align*} We define 
\begin{align*}
    f_{j_n i}=\bigg[f_{j_n}-\frac{1}{m(R_{j_n i})} \int_{R_{j_n i}} f_{j_n} \ dm \bigg] \chi_{R_{j_n i}}, \quad g_{j_n}= f_{j_n} - \sum_{i \in \mathbb{N}} f_{j_n i}.
\end{align*}
Then, arguing as above, $a_{j_n}=(\widetilde{C}^{1/p}\alpha^{(n+1)p} m(R_{j_n}))^{-1}g_{j_n}$ is a $(1,\infty)$-atom, $f_{j_n i}$ is supported in $R_{j_n i}$ and has zero integral, 
\begin{align*}
    \bigg( \frac{1}{m(R_{j_n i})} \int_{R_{j_n i}}|f_{j_n}|^p \ dm \bigg)^{1/p} \le \widetilde{C}^{1/p}\alpha^{n+1} <  {2}\widetilde{C}^{1/p}\alpha^{n+1},
\end{align*}
and
 \begin{align*}  
    |f_{j_n i}(x)| \le |f_{j_n}(x)|+ \widetilde{C}^{1/p}\alpha^{n+1}  \le |b(x)|+\widetilde{C}^{1/p}\alpha^n2^{n-1}+\widetilde{C}^{1/p}\alpha^{n+1}  \le
     |b(x)|+\widetilde{C}^{1/p}\alpha^{n+1}2^n. 
\end{align*}
     We deduce that
     \begin{align*}
          \sum_{j_n i} \|f_{j_n i}\|_p^p &\le \sum_{j_n} 2^p \|f_{j_n}\|_p^p \le 2^{p(n+1)}\|b\|_p^p, \\ 
     \sum_{j_n i}m(R_{j_n i}) &\le \frac{1}{\alpha^{(n+1)p}} \sum_{j_n}\sum_i \int_{R_{j_n i}}  |f_{j_n}|^p \ dm 
    \le \frac{1}{\alpha^{(n+1)p}} \sum_{j_n}\|f_{j_n}\|_p^p \le \frac{1}{\alpha^{(n+1)p}}2^{pn}\|b\|_p^p
     \end{align*}
    and this concludes the proof.
  \end{proof}
  In the sequel we write $H^{1}(m)$ in place of $H^{1,\infty}(m)$ and $H^1_{\text{fin}}(m)$ in place of $H^{1,\infty}_{\text{fin}}(m)$. \\ 
  \begin{oss}
 We now show that the Hardy space $H^1(m)$ does not depend on the choice of $\beta.$ \\ 
Fix $12 \le \beta<\beta'$ and denote by $\mathcal{R}$ and $\mathcal{R}'$ the family of admissible trapezoids corresponding to the parameters $\beta$ and $\beta'$, respectively. We call $H^1_{\beta}(m)$ and $H^1_{\beta'}(m)$ the correspondent Hardy spaces with atoms supported in  sets  in $\mathcal{R}$ and $\mathcal{R}'$  respectively. It is clear that $H^1_{\beta}(m) \subset H^1_{\beta'}(m)$. For the reverse inclusion, we prove that any $(1,\infty)$-atom in $H^1_{\beta'}(m)$ can be decomposed as the sum of multiples of $(1,\infty)$-atoms in $H^1_{\beta}(m)$ in such a way that the norm is uniformly bounded. \\ 
First assume $\beta' \le 2 \beta$. Consider a $(1,\infty)$-atom $a \in H^1_{\beta'}(m)$ supported in a set $R=R_{h'}^{h''}(x) \in \mathcal{R}' \setminus \mathcal{R}.$ \\ 
By applying the decomposition algorithm to $R$ we obtain $R_1=R_{h'}^{2h'}(x)$ and $R_2=R_{2h'}^{h''}(x).$  Now we call $T=R_{2h'}^{4h'}(x),\ T_1=R_{h'}^{4h'}(x), \ T_2=R_2$. Obviously $R_1, R_2, T, T_1 \in \mathcal{R}.$ We define
\begin{align*}
    \varphi_i= a \chi_{R_i}- \bigg(\frac{1}{m(T)} \int_V a \chi_{R_i} \ dm \bigg) \chi_{T}, \quad i=1,2.
\end{align*}
We have that $\int_{V} \varphi_i \ dm=0$ for $i=1,2$ and $\varphi_1+\varphi_2=a$ as consequence of the vanishing integral of $a$. Moreover,  
\begin{align*}
    \| \varphi_i\|_\infty \le 2 \|a\|_\infty \le \frac{2}{m(R)} \le \frac{2}{m(T_i)},
\end{align*}
for $i=1,2.$ Observe that $\varphi_i$ is supported in $R_i \cup T =T_i$ because $4h' <h''.$ Thus $\varphi_i/2$ is a $(1,\infty)$-atom supported in $T_i \in \mathcal{R}$ and $H^1_{\beta'}(m) \subset H^1_{\beta}(m)$. \\  %In conclusion, take any function $g \in H^1_{\beta'}$ and consider a decomposition 
%\begin{align*}
 %   g= \sum_i \lambda_i a_i.
%\end{align*} We can rewrite every $a_i$ as the sum of two $(1,\infty)$ atoms in $H^1_\beta(m)$ times 2. Thus $\|g\|_{H^1_\beta} \le 4 \|g\|_{H^1_{\beta'}}.$ \\ This concludes the proof because, 
Suppose now that $2^{n-1} \beta \le \beta' \le 2^n \beta$ for some $n >1$. We observe that $H^1_{\beta'}(m) = H^1_{\beta'/2}(m) = H^1_{\beta'/4}(m)\dots = H^1_{\beta'/2^n}(m).$  Thus $H^1_{\beta'}(m)=H^{1}_{\beta}(m)$ and this concludes the proof.  \qed
  \end{oss}
  
\begin{oss} If we choose $T=\mathbb{T}_{q+1}$ and $\mu(\cdot)=q^{\ell(\cdot)}$ as a particular flow measure,  it can be used a similar argument in order to show the equivalence of $H^1_{\beta}(\mu)$ and the space $H^1(\mu)$ introduced in \cite{ATV1}.
\end{oss}
  \subsection{Duality between $H^{1}(m)$ and $BMO(m)$} We now establish the duality between $H^1(m)$ and $BMO(m)$. We first need a lemma which provides a covering of $V$ made by an increasing family of admissible trapezoids.
  \begin{lem}\label{TR}
There exists a family   $\{R_j\}_j \subset \mathcal{R}$ such that $R_j \subset R_{j+1}$ and $\cup_j R_j=V$.
\end{lem}
\begin{proof}
Let $R_0=\lbrace x_0\rbrace$, and define $R_j$ to be the output of the expansion algorithm applied to $R_{j-1}$ for $j\geq 1$. About vertical expansions, choose at random whether to expand up or down for the first occurring one (which is the one producing $R_4$ out of $R_3$) and then always alternate them (for example, if we decide to extend $R_3$ down, the next vertical expansion will be up). Observe that a vertical expansion always needs to be followed by a vertical one. The opposite is not true, but still horizontal and vertical expansions will definitely alternate since, for any given $R_{h'}^{h''}(x_0)\in \mathcal{R}$, it holds $\frac{h''+k}{h'+k}<3$ for $k$ large enough. It is then clear that $R_j\subset R_{j+1}$ and $\cup_{j\geq 0}R_j=V$.
\end{proof}
\begin{thm}[Duality between $H^1(m)$ and $BMO(m)$] 
  (i) Suppose $f \in BMO(m)$. Then the linear functional $\ell$ given by  \begin{align*}
        \ell(g)= \int_V fg\  dm, 
    \end{align*} initially defined on the dense subspace $H^1_{\mathrm{fin}}(m)$, has a unique bounded extension to $H^1(m)$ and there exists $C>0$ such that
    \begin{align*}
        \|\ell\|_{(H^1)'} \le C \|f\|_{BMO}.
    \end{align*}
 (ii)   Conversely, every continuous linear functional $\ell$ on $H^1(m)$ can be realized as above, with $f \in BMO(m)$, and there exists $C>0$ such that
 \begin{align*}
     \|f\|_{BMO} \le C \| \ell\|_{(H^1)'}.
 \end{align*}
\end{thm}
\begin{proof} For the proof of $(i)$ we can closely follow  \cite{FS, gra} for the Euclidean setting.
\begin{comment}
 $(i)$ We show that 
\begin{align}\label{4}
\bigg| \int_V fg\ dm \ \bigg| \le c \|f\|_{BMO_1}\|g\|_{H^1},
\end{align}
for $f \in BMO_1(m)$ and $g \in H^1_{\text{fin}}(m)$.
Assume initially that $f$ is bounded. Then 
\begin{align*}
   \int_V fg  dm = \sum_{k} \lambda_k \int_V fa_k \ dm,
\end{align*}
where $g \in H^1$ and $\sum_k \lambda_k a_k$ is an atomic decomposition. Since $a_k$ has vanishing mean value,
\begin{align*}
    \int_V fa_k\  dm= \int_{R_k} [f-f_{R_k}]a_k\ dm,
\end{align*}
where $a_k$ is supported in $R_k$. \\ 
Since $|a_k(x)| \le |R_k|^{-1}$, we have that 
    \begin{align*}
\bigg| \int_V fg\  dm \bigg | &\le \sum_k \frac{|\lambda_k|}{|R_k|} \int_{R_k} |f-f_{R_k}|\ dm  \\ &\le \sum_{k} |\lambda_k| \| f\|_{BMO_1}.       
    \end{align*}
    This proves \eqref{4} for $f$ bounded and $g \in H^1.$ To conclude assume that $f \in BMO_1(m)$ is real valued  and $g \in H^1_{\text{fin}}$. Replace $f$ with 
    \begin{align*}
        f^{(k)}(x)=\begin{cases} -k &\text{if $f(x) \le -k$,} \\ 
        f(x) &\text{if $-k \le f(x) \le k$,} \\ 
        k &\text{if $f(x) \ge k.$}
        \end{cases}
    \end{align*}
    Notice that $f^{(k)}= \min(\max(f,-k),k),$ thus $\|f^{(k)}\|_{BMO_1} \le c \| f\|_{BMO_1}$ and $$\bigg|\int_V f^{(k)}g  \  dm \bigg| \le c \|f\|_{BMO_1}\|g\|_{H^1}.$$ 
     Now we use that $f^{(k)} \to f$ pointwise and that $g \in H^1_{\text{fin}}(m)$ in order to apply the dominated convergence theorem and  \begin{align*}
       \bigg| \int_V fg\ dm \ \bigg| \le c \|f\|_{BMO_1}\|g\|_{H^1}.
   \end{align*} 
\end{comment}
 
 We prove $(ii)$. For every $R \in  \mathcal{R}$ we denote by $L^2_R$ the space of all square summable functions supported in $R$ with norm ${L^2}$ and by $L^2_{R,0}$ its closed subspace of function with integral zero. Note that $g \in L^2_{R,0}$ implies that $g$ is a multiple of a $(1,2)$-atom and that $\|g\|_{H^1} \le A_2m(R)^{1/2} \|g\|_{L^2}.$ Thus, if $\ell$ is a given functional on $H^1(m)$ then $\ell$ extends to a linear functional on $L^2_{R,0}$ with norm at most $A_2m(R)^{1/2}\|\ell\|_{(H^1)'}$ by Proposition \ref{EqH}.
 
Since the dual of $L^2_{R,0}$ is the quotient of $L^{2}_{R}$ module constant functions, by the Riesz theorem, there exists a unique $F^R$ in $L^2_{R}$ module constant functions such that
   \begin{align*}
       \ell(g)= \int_R F^{R}g\  dm, \ \forall g \in L^2_{R,0}, \quad \text{and} \quad \|F^R\|_{L^2(R)} \le A_2m(R)^{1/2} \| \ell\|_{(H^1)'}.
   \end{align*}
    Observe that if $R \subset R'$ then $F^R-F^{R'}$  is constant on $R.$ Let ${R}_j$ as in Lemma \ref{TR} for $j=0,1,2,... .$ Define  $f :V \to \mathbb{C}$ by setting 
    \begin{align*}
        f(x)=F^{R_j}(x)- \frac{1}{m(R_1)}\int_{ R_1} F^{R_j} \ dm
    \end{align*}
    whenever $x \in {R}_j$. It is easy to verify that the definition of $f$ in unambiguous and $f \in BMO(m)$.  \end{proof} 

\section{Interpolation and integral operators}\label{s: int}
We will prove here some interpolation results involving  Hardy and $BMO$ spaces. The real interpolation results will be essentially a consequence of the Calder\'on--Zygmund decomposition that we constructed in Section \ref{subs: CZ}. To obtain complex interpolation results we will need to discuss  the sharp maximal function associated with the admissible trapezoids.  

\subsection{Real interpolation}

In this subsection we study the real interpolation of $H^1(m)$, $BMO(m)$ and the $L^p(m)$ spaces. We refer the reader to \cite{Jo} for an overview of the real interpolation results which hold in the classical setting. Our aim is to prove similar results in our context. Note that in our case a maximal characterization of $H^1(m)$ is not available, so that we cannot follow the classical proofs but we shall only exploit the atomic definition of $H^1(m)$. We also notice that the proofs of the our results follow closely those of \cite[Section 5]{V}. 

We first recall some notation of the real interpolation of normed spaces, focusing on the $K$-method. For the details see \cite{BL}. 

Given two compatible normed spaces $A_0$ and $A_1$, for any $t>0$ and for any $a\in A_0+A_1$ we define
$$K(t,a;A_0,A_1)=\inf\{ \|a_0\|_{A_0}+t\|a_1\|_{A_1}:~a=a_0+a_1,\,a_i\in A_i \}\,.$$ 
Take $q\in [1, \infty]$ and $\theta\in (0,1)$. The {\emph{real interpolation space}} $\big[A_0,A_1\big]_{\theta,q}$ is defined as the set of the elements $a\in A_0+A_1$ such that
$$\|a\|_{\theta,q}=\begin{cases}
\Big(\int_0^{\infty}\big[t^{-\theta}\,K(t,a;A_0,A_1)\big]^q \frac{\di t}{t} \Big)^{1/q}&{ \text{if~}} 1\leq q<\infty\,,\\
\|t^{-\theta}\,K(t,a;A_0,A_1)\|_{\infty}&{ \text{if~}} q=\infty\,,
\end{cases}
$$
is finite. %The space $\big[A_0,A_1\big]_{\theta,q}$ endowed with the norm $\|a\|_{\theta,q}$ is an exact interpolation space of exponent $\theta$. 

\smallskip
We shall first estimate the $K$ functional of $L^{p}$-functions with respect to the couple of spaces $(H^1(m),L^{p_1}(m))$, $1<p_1\leq \infty$.
\begin{lem}\label{intpinfty}
	Suppose that $1<p< p_1\leq \infty$ and let $\theta\in (0,1)$ be such that $\frac{1}{p}=1-\theta+\frac{\theta}{p_1}$. Let $f$ be in $L^p(m)$. The following hold:
	\begin{enumerate}
		\item[(i)] there exist positive constants $D_1, D_2$ such that for every $\lambda>0$ there exists a decomposition $f=g^{\lambda}+b^{\lambda}$ in $L^{p_1}(m)+H^1(m)$ such that
		\begin{enumerate}
			\item[(i1)] $\|g^{\lambda}\|_{\infty} \leq \widetilde C^{1/p}\,\lambda$ and, if $p_1<\infty$,  $\|g^{\lambda}\|_{p_1}^{p_1}\leq D_1\,\lambda^{p_1-p}\,\|f\|_{p}^p$;
			\item[(i2)] $\|b^{\lambda}\|_{H^1}\leq D_2\,\lambda^{1-p}\,\|f\|_{p}^p\,;  $
		\end{enumerate}
		\item[(ii)] there exists a constant $K_p>0$ such that
		\begin{enumerate}
			\item[(ii1)] for any $t>0$, $K(t,f;H^1(m),L^{p_1}(m))\leq K_p\,t^{\theta}\,\|f\|_{p};$
			\item[(ii2)] $f\in  [H^1(m),L^{p_1}(m)]_{\theta,\infty}$ and $\|f\|_{\theta,\infty}\leq K_p\,\|f\|_{p}.$ 
		\end{enumerate}
	\end{enumerate}
\end{lem}
\begin{proof}

	Let $f$ be in $L^p(m)$.
	We first prove $(i)$. Given a positive $\lambda$, let $\{R_i\}_i \subset \mathcal{R}$ be the collection of admissible trapezoids associated with the Calder\'on--Zygmund decomposition of $|f|^p$ corresponding to the value $\lambda ^p$. We write  
	$$
	f= g^{\lambda}+b^{\lambda}=g^{\lambda}+\sum_ib^{\lambda}_i=g^{\lambda}+\sum_i(f-f_{R_i})\chi_{R_i}\,.
	$$
	We then have
	$$\|g^{\lambda}\|_{{\infty}}\leq {\widetilde C}^{1/p}\,\lambda, \qquad  \frac{1}{m(R_i)}\int_{R_i}|f|^pd m \le   \widetilde{C} \lambda^p \qquad{ \text{and}}\qquad |f_{R_i}|\leq \widetilde C^{1/p} \lambda .$$
	If $p_1<\infty$, then
	\begin{align*}
	\|g^{\lambda}\|_{{p_1}}^{p_1}&\leq \int_{(\bigcup  R_i)^c}|f|^{p_1}dm+\sum_i\int_{ R_i}|f_{R_i}  |^{p_1}dm\\
		&\leq  \int_{(\bigcup  R_i)^c}|f|^{p_1-p}|f|^pdm+\widetilde C^{p_1/p}\lambda^{p_1}\sum_{R_i}m(R_i)\leq D_1\lambda^{p_1-p}\|f\|_{p}^p\,.
\end{align*}

		Thus $(i1)$ holds.
	Concerning $(i2)$, for any $i$, $b_i^{\lambda}$ is supported in $R_i$, has vanishing integral and
	$$
	\|b_i^{\lambda}\|_{p} \lesssim  m(R_i)^{1/p} \lambda \lesssim \lambda m(R_i)m(R_i)^{-1+1/p} \,.
	$$
	This shows that $b_i^{\lambda}\in H^{1,p}(m)$ and $\|b_i^{\lambda}\|_{H^1}\lesssim\,\lambda\,m( R_i)$. Since $b^{\lambda}=\sum_ib_i^{\lambda}$, $b^{\lambda}$ is in $H^{1}(m)$ and 
	$$\|b^{\lambda}\|_{H^{1}}\lesssim\,\lambda\,\sum_im( R_i)\leq D_2\,\lambda\,\frac{\|f\|_{p}^p}{\lambda^p}\,,$$
	as required.
	
	We now prove $(ii)$. Fix $t>0$. For any positive $\lambda$, let $f=g^{\lambda}+b^{\lambda}$ be the decomposition of $f$ in $L^{p_1}(m)+H^1(m)$ given by $(i)$. Thus 
	\begin{align*}
	K(t,f;H^1(m),L^{p_1}(m))&\leq \inf_{\lambda>0} \big( \|b^{\lambda}\|_{H^1}+t\,\|g^{\lambda}\|_{{p_1}} \big)
	\lesssim\,\inf_{\lambda>0}\big(\lambda^{1-p}\,\|f\|_{p}^{p}+t\,\lambda^{1-p/p_1}\|f\|_{p}^{p/p_1}  \big)\,.
	%&\leq C\,\|f\|_{L^p}^{p/{p_1}}\,\inf_{\lambda>0}\big(\lambda^{1-p}\,\|f\|_{L^p}^{p(1-1/{p_1})}+t\,\lambda^{1-p/p_1}  \big)\\
	%&=C\,\|f\|_{p}^{p/{p_1}}\,\inf_{\lambda>0} G(t,\lambda)\,,
	\end{align*}
	Arguing as in \cite[p. 292]{V} it follows that there exists a positive constant $K_p$ such that 	
	$$K(t,f;H^1(m),L^{p_1}(m))\leq K_p\,\|f\|_{p}\,t^{\theta}\,,$$
	proving $(ii1)$. This implies that $$\|t^{-\theta}\,K(t,f;H^1(m),L^{p_1}(m))\|_{{\infty}}\leq K_p\,\|f\|_{p},$$ so 
	that $f\in [H^1(m),L^{p_1}(m)]_{\theta,\infty}$ and $\|f\|_{\theta,\infty}\leq K_p\|f\|_{p}$, as required in $(ii2)$.
	
\end{proof}
Mimicking the proofs of \cite[Th. 5.2, Cor. 5.4 and 5.7, Prop. 5.6]{V}, we deduce from Lemma \ref{intpinfty} the subsequent results.
\begin{thm}\label{realintH1Lp2}
The following hold:
\begin{itemize}
\item[(i)]  Let $1<p<p_1\leq \infty$ and $\theta\in(0,1)$ be such that $\frac{1}{p}=1-\theta+\frac{\theta}{p_1}$. Then 
	$$\big[H^1(m),L^{p_1}(m)  \big]_{\theta,p}=L^p(m)\,.$$
	\item[(ii)] Let $1\leq q_1<q<\infty$  and $\frac1q=\frac{1-\theta}{q_1}$, with $\theta\in (0,1)$. Then 
	$$\big[L^{q_1}(m),BMO(m)\big]_{\theta,q}=L^q(m)\,.$$
	\item[(iii)] Let $1<q<\infty$  and $\frac1q={1-\theta}$, with $\theta\in (0,1)$. Then 
	$$\big[H^{1}(m),BMO(m)\big]_{\theta,q}=L^q(m)\,.$$

	\end{itemize}
\end{thm}

\subsection{Sharp maximal function}
In this section we introduce the sharp maximal function of a function $f: V \to \mathbb{C}$, defined by
\begin{align*}
    M^{\#}f(x)= \sup_{R \ni x} \frac{1}{m(R)}\int_{ R}|f-f_R|\  dm,
\end{align*} where the supremum is taken over all $R \in  \mathcal{R}$ such that $x \in R.$ The sharp maximal function is a useful tool to capture the local behaviour of the mean oscillation of a function. Obviously, we have $\|M^{\#}f\|_\infty=\|f\|_{BMO_1}$ and $M^{\#}f \le 2Mf$ pointwise. By the boundedness of the Hardy--Littlewood maximal function, we easily conclude that, for all $p \in (1,+\infty]$,
\begin{align}\label{SMF}
    \|M^{\#}(f)\|_p \le M_p \|f\|_p, \qquad{\forall f \in L^p(m)},
\end{align}
for some $M_p$ depending only on $p$. \\ 
%As a by-product of this result, we provide a  complex interpolation theorem involving $BMO(m)$.
%\subsection{Dyadic family}

Now we prove the existence of a dyadic family of partitions of the set of vertices of the tree consisting of admissible trapezoids, by which we will obtain the converse  inequality to \eqref{SMF}. We remark that, in a certain  sense, such a family is the analogue of the classical family of Euclidean  dyadic cubes. Indeed, we shall prove that they share similar properties of set inclusion and of measure comparability. The strategy to obtain the dyadic sets  is based on the decomposition and expansion algorithms. \begin{thm}\label{dyadic}
There exists a family $\lbrace \mathcal{D}_j\rbrace_{j\in \mathbb{Z}}$ of partitions of $V$ consisting of admissible trapezoids such that
\begin{itemize}
    \item[(i)] $R\subset R'$ or $R\cap R'=\emptyset$ whenever $R\in \mathcal{D}_j, R'\in \mathcal{D}_k$, $k>j$.
    \item[(ii)] For any $j\in \mathbb{Z}$ and $R\in \mathcal{D}_j$, there exists a unique $R'\in \mathcal{D}_{j+1}$ such that $R \subset R'$ and $m(R')\leq \widetilde C m(R)$. %where $C_0=\max\{\overline{c},\widetilde{c}\}$ and $\overline{c}, \widetilde{c}$ are chosen as in Proposition \ref{costanti}.
    \item[(iii)] For every $j\in \mathbb{Z}$, $R\in \mathcal{D}_j$ can be written as the disjoint union of at most $c$ elements of $\mathcal{D}_{j-1}$, where $c$ is the constant in \eqref{control sons}.
    \item[(iv)] For all $x\in V$ there exists $k=k(x)\in \mathbb{Z}$ such that $x\in \mathcal{D}_j$ for all $j\leq k$.
\end{itemize}
\end{thm}
\begin{proof}
{Let $\lbrace R_j\rbrace$ be the family of trapezoids described in Lemma \ref{TR}.}
 For each $j\geq 0$, $R_j$ can be used as a base set to produce a partition of $V$. Let $h'(j), h''(j)$ be the parameters defining $R_j$. Given a trapezoid $R$, we write $\mathcal{B}(R)$ for the \textit{brothers} of $R$, the family of trapezoids of parameters $h'(j), h''(j)$ and root of the same level as $x_j$, the root of $R_j$. A partition of the strip of vertices $\lbrace x\in V: \ \ell(x_j)-h''(j)< \ell(x)\leq \ell(x_j)-h'(j)\rbrace$ is naturally given by $\mathcal{B}(R_j)$. Let $\mathcal{L}$ be the set of indices such that $R_\ell$ is obtained by vertical expansion of $R_{\ell-1}$ when $\ell\in \mathcal{L}$. It is easily seen that $R_\ell\setminus R_{\ell-1}$ is still an admissible trapezoid, and consequently are all its brothers. For $j\in \mathbb{N}$ we set $\mathcal{D}_j$ to be the collection of all trapezoids $R$ belonging to $\mathcal{B}(R_j)$ or to $\mathcal{B}(R_\ell\setminus R_{\ell-1})$ for some $\ell>j$. Then it is clear that $\mathcal{D}_j$ defines a partition of $V$.
For $j<0$, we define $\mathcal{D}_j$ to be the family of trapezoids obtained by applying one step of the decomposition algorithm to each trapezoid $R\in \mathcal{D}_{j-1}$. Then the family of partitions $\lbrace \mathcal{D}_j\rbrace_{j\in \mathbb{Z}}$ satisfies all the desired properties: $(i)$ and $(ii)$ follow from the rules defining the expansion algorithm, $(iii)$ and $(iv)$ from the ones of the decomposition algorithm and Corollary \ref{corollario bounded degree}.
\end{proof}
We set $\mathcal{D}= \cup_{j \in \mathbb{Z}} \mathcal{D}_j $ and we define the maximal dyadic function of $f :V\to \mathbb{C}$ as
\begin{align*}
    M_{\mathcal{D}}f(x)= \sup_{R\ni x} \frac{1}{m(R)}\int_{ R} |f| \ dm,
\end{align*}
where the supremum is taken over all $R \in \mathcal{D}$ such that $x \in R.$ We remark that $M_\mathcal{D}(f) \le M(f)$ pointwise on $V$, thus $M_\mathcal{D}$ is of  weak-type (1,1). In this section  we shall prove that the functions $Mf$, $M^{\#}f$,  $M_{\mathcal{D}}f$ and $f$ are comparable in the $L^p$ norm. 

In the proof of the next theorem we follow a standard argument, see for example \cite[Th. 7.4.4.]{gra} for the correspondent result in the Euclidean setting. 
\begin{thm}\label{ge} For all $\gamma>0$, $\lambda>0$ and $f :V\to \mathbb{C}$, it holds
\begin{align*}
    m(\{ x \in V : M_{\mathcal{D}} (f)(x)> 2\lambda, M^{\#}(f)(x)<\gamma\lambda \}) \le C'\gamma m(\{x \in V : M_{\mathcal{D}}(f)(x)>\lambda\}),
\end{align*}
where $C'=2\beta \widetilde C$.% with $C_0$ as in Theorem \ref{dyadic}.
\end{thm}
\begin{proof}
We can assume that $\Omega_{\lambda}=\{x \in V : M_{\mathcal{D}}(f)(x)>\lambda\}$ has finite measure. Hence for all $x \in \Omega_\lambda$, there exists $R_x \in \mathcal{D}$ that is maximal with respect to set inclusion, such that $x \in R_x$ and 
\begin{align*}
    \frac{1}{m(R_x)}\int_{R_x}|f|\ dm > \lambda.
\end{align*}
Notice that if $y \in R_x$, then $R_x=R_y$ because maximal trapezoids are disjoint. Hence it is sufficient to show that for a given $R_x$
\begin{align}\label{suff}
    m(\{ y \in R_x : M_{\mathcal{D}} f(y)> 2\lambda, M^{\#}f(y)<\gamma\lambda \}) \le C'\gamma m(R_x).
\end{align}
Fix $x$ and $y \in R_x$ such that $M_\mathcal{D}(f)(y)> 2 \lambda$, then the supremum
\begin{align*}
    \sup_{ R \ni y} \frac{1}{m(R)} \int_R |f|\ dm
\end{align*}
is taken over all the dyadic trapezoids $R$ which contain $R_x$ or are contained in $R_x$. If $Q$ strictly contains $R_x$, then by the maximality of $R_x$, we get
\begin{align*}
    \frac{1}{m(Q)} \int_Q |f| \ dm\le \lambda.  
\end{align*}
Thus 
\begin{align*}
    M_{\mathcal{D}}f(x)= M_\mathcal{D}(f \chi_{R_x})(x).
\end{align*}
Let $R'_x$ be the dyadic set of minimal scale such that $R_x \subsetneq R'_x$.  It follows
\begin{align*}
 M_{\mathcal{D}}\bigg((f-f_{R'_x})\chi_{R_x})\bigg)(x) >M_{\mathcal{D}}(f\chi_{R_x})(x)- |f_{R_x'}| > 2\lambda - \lambda= \lambda. 
\end{align*}
We conclude that
\begin{align}\label{espressione}
    m(\{y \in R_x : M_{\mathcal{D}} f(y)> 2\lambda\}) \le m(\{y \in R_x : M_{\mathcal{D}}\bigg((f-f_{R'_x})\chi_{R_x})\bigg)(y) >\lambda\}).
\end{align}
We know that $M_{\mathcal{D}}$ is of weak type (1,1), thus we control the last expression in \eqref{espressione} with
\begin{align}\label{fine}
    \frac{2\beta}{\lambda} \int_{R_x} |f-f_{R'_x}|\ dm \le \frac{2\widetilde C\beta}{\lambda} \frac{m(R_x)}{m(R'_x)} \int_{R'_x} |f-f_{R'_x}| \ dm \le \frac{2\widetilde C\beta}{\lambda}m(R_x)M^{\#}(f)(\xi_x),
\end{align}
where $\xi_x \in R_x.$ We can assume that for some $\xi_x \in R_x$ it holds
\begin{align*}
M^{\#}(f)(\xi_x) \le \gamma \lambda,
\end{align*} otherwise there is nothing to prove. This, together with \eqref{suff}, \eqref{espressione} and \eqref{fine} conclude the proof. 
\end{proof}
Now we prove an inequality involving the $L^p$ norm of a function,  the dyadic and the sharp maximal functions. 
\begin{thm}\label{INT} Let $0<p_0<+\infty.$ Then, for all $p$ such that $p_0 \le p < + \infty$, there exists a constant $N_p$ such that, for all  $f$ with $M_\mathcal{D}(f) \in L^{p_0}(m)$, we have
\begin{itemize}
\item[(i)] $    \|M_{\mathcal{D}}(f)\|_p \le N_p \| M^{\#}(f)\|_p$;
\item[(ii)]    $\|f\|_p \le N_p \| M^{\#}(f)\|_p$. 
\end{itemize}
\end{thm}
\begin{proof}
To prove $(i)$ it is possible to reply step by step \cite[Th. 7.4.5.]{gra}, so we omit the details. Inequality $(ii)$ follows from $(i)$ and the pointwise estimate $|f| \le M_\mathcal{D}(f)$.
\end{proof}

\subsection{Complex interpolation}

We study the complex interpolation of $H^1(m)$, $BMO(m)$ and the $L^p(m)$ spaces.  \\
Given two compatible normed spaces $A_0$ and $A_1$, for any $\theta\in (0,1)$ we denote by $\big(A_0,A_1\big)_{[\theta]}$ the complex interpolation space obtained via Calder\'on's complex interpolation method (see \cite{BL} for the details). More precisely, we denote by $\Sigma$ the strip $\{z\in\mathbb C: \Re z\in (0,1)\}$ and denote by $\overline{\Sigma}$ its closure. We consider the class $\mathcal F(A_0,A_1)$ of all functions $F:\overline{\Sigma}\rightarrow A_0+A_1$ such that the map $z\mapsto \langle F(z),\ell \rangle$ is continuous and bounded in $\overline{\Sigma}$ and analytic in $\Sigma$ for every $\ell$ in the dual of $A_0+A_1$, $F(it)$ is bounded on $A_0$ and $F(1+it)$ is bounded on $A_1$ for every $t\in\mathbb R$. We endow $\mathcal F(A_0,A_1)$ with the norm 
$$
\|F\|_{\mathcal F}=\sup\{\max(\|F(it)\|_{A_0}, \|F(1+it)\|_{A_1}):t\in\mathbb R \}\,.
$$
The space $\big(A_0,A_1\big)_{[\theta]}$ consists of all $f\in A_0+A_1$ such that  $f=F(\theta)$ for some $F\in \mathcal F(A_0,A_1)$ endowed with the norm
$$
\|f\|_{[\theta]}=\inf \{ \|F\|_{\mathcal F}: F\in \mathcal F(A_0,A_1), F(\theta)=f\}.
$$

\begin{thm}\label{t: intcomp}
 Suppose that $\theta\in (0,1)$, $1<q_1<q<\infty$, $1<p<p_1<\infty$, $\frac{1}{q}=\frac{1-\theta}{q_1}$ and $\frac{1}{p}=1-{\theta}+\frac{\theta}{p_1}$. Then the following hold:
 \begin{itemize} 
 \item[(i)] $\big(L^{q_1}(m),BMO(m)\big)_{[\theta]}=L^q(m)$;
 \item[(ii)] $\big(H^1(m), L^{p_1}(m)\big)_{[\theta]}=L^p(m)$.
 \end{itemize}
 \end{thm}
 \begin{proof}
 We first prove $(i)$. The inclusion $L^q(m)\subset \big(L^{q_1}(m),BMO(m)\big)_{[\theta]}$ follows from the fact that $L^{\infty}(m)$ is continuously included in $BMO(m)$ and $L^q(m)=\big(L^{q_1}(m),L^{\infty}(m)\big)_{[\theta]}$. To prove the reverse inclusion we consider any function $\phi:V\rightarrow \mathcal R$ which associates to every vertex $x\in V$ an admissible trapezoid $R$ which contains $x$ and any function $\eta:V\times V\rightarrow \mathbb C$ such that $|\eta(x,y)|=1$ for every $(x,y)\in V\times V$. Define the linear operator $S^{\phi,\eta}$ which on every function $f:V\rightarrow \mathbb C$ is defined as follows 
 $$
 S^{\phi,\eta}f(x)=\frac{1}{m(\phi(x))}\int_{\phi(x)}[f(y)-f_{\phi(x)}]\eta(x,y)dm(y),\qquad \forall x\in V\,.
 $$
 Clearly, 
 $$
 |S^{\phi,\eta}f|\leq M^{\sharp}(f),\qquad { \text{and}}\qquad \sup_{\phi, \eta}|S^{\phi,\eta}f|=M^{\sharp}(f)\,.
 $$
%We denote by $\Sigma$ the strip $\{z\in\mathbb C: \Re z\in (0,1)\}$ and denote by $\overline{\Sigma}$ its closure. We consider the class $\mathcal F(L^{p_1}(m),BMO(m))$ of all functions $F:\overline{\Sigma}\rightarrow L^{q_1}(m)+BMO(m)$ such that the map $z\mapsto \langle F(z),\ell \rangle$ is continuous and bounded in $\overline{\Sigma}$ and analytic in $\Sigma$ for every $\ell$ in the dual of $L^{q_1}(m)+BMO(m)$ and $F(it)$ is bounded on $L^{q_1}(m)$ and $F(1+it)$ is bounded on $BMO(m)$ for every $t\in\mathbb R$. We endow $\mathcal F(L^{q_1}(m),BMO(m))$ with the norm 
%$$
%\|F\|_{\mathcal F}=\sup\{\max(\|F(it)\|_{q_1}, \|F(1+it)\|_{BMO}):t\in\mathbb R \}\,.
%$$
Given $f\in \big(L^{q_1}(m),BMO(m)\big)_{[\theta]}$ there exists $F\in \mathcal F(L^{q_1}(m),BMO(m))$ such that $F(\theta)=f$. For every $t\in\mathbb R$ we have
$$
\|S^{\phi,\eta}F(it)\|_{q_1}\leq \|M^{\sharp}(F(it))\|_{q_1}\lesssim \|MF(it)\|_{q_1}\lesssim\|F(it)\|_{q_1}\,,
$$
where $M$ is the Hardy--Littlewood maximal function associated with $\mathcal R$ which is bounded on $L^{q_1}(m)$. We also have that for every $t\in\mathbb R$
$$
\|S^{\phi,\eta}F(1+it)\|_{\infty}\leq \|M^{\sharp}(F(1+it)) \|_{\infty}\leq  \|F(1+it)\|_{BMO} \,.
$$
It follows that $S^{\phi,\eta}F\in \mathcal F(L^{q_1}(m),BMO(m))$ and $\|S^{\phi,\eta}F\|_{\mathcal F}\lesssim \|F\|_{\mathcal F}$. Hence 
$$
\|S^{\phi,\eta}F(\theta)\|_{q}\lesssim\|F(\theta)\|_{(L^{q_1}(m), BMO(m))_{[\theta]}}= \|f\|_{(L^{q_1}(m), BMO(m))_{[\theta]}}\,.
$$
By taking the supremum over all $\phi$ and $\eta$ and applying Theorem \ref{INT} we get
$$
\|f\|_q\lesssim\|M^{\sharp} (f)\|_q\lesssim\|f\|_{(L^{q_1}(m), BMO(m))_{[\theta]}}\,.
$$
This proves the inclusion $\big(L^{q_1}(m),BMO(m)\big)_{[\theta]}\subset L^p(m)$ and concludes the proof of $(i)$. The proof of $(ii)$ follows by the duality between $H^1(m)$ and $BMO(m)$ and \cite[Cor. 4.5.2]{BL}. 
\end{proof}

 \begin{thm}
  Suppose that $\theta\in (0,1)$, $\frac{1}{q}={1-\theta}$. Then the following hold:
 \begin{itemize} 
 \item[(i)] $\big(L^{1}(m),BMO(m)\big)_{[\theta]}=L^q(m)$;
 \item[(ii)] $\big(H^1(m), L^{\infty}(m)\big)_{[\theta]}=L^q(m)$;
 \item[(iii)] $\big(H^1(m), BMO(m)\big)_{[\theta]}=L^q(m)$.
 \end{itemize}
 \end{thm}
 \begin{proof}
 Take $r\in (1,q)$ and $\phi\in (0,1)$ such that $\frac1r=1-\phi+\frac{\phi}{q}$. Then $\big(L^1(m),L^q(m)\big)_{[\phi]}=L^r(m)$. Moreover, by Theorem \ref{t: intcomp},
$$\big(L^{r}(m),BMO(m)\big)_{[\gamma]}=L^q(m)$$ if $\frac{1}{q}=\frac{1-\gamma}{r}$. 
Since $L^1(m)\cap BMO(m)$ contains the space of compactly supported functions, it is dense in $L^r(m)$ and $L^q(m)$. Then, by the reiteration theorem \cite[Th. 2]{wolff}, we deduce that $\big(L^{1}(m),BMO(m)\big)_{[\theta]}=L^q(m)$.  
 
Take $r\in (1,q)$ and $\phi\in (0,1)$ such that $\frac1r=1-\phi+\frac{\phi}{q}$. Then by Theorem \ref{t: intcomp}, $\big(H^1(m),L^q(m)\big)_{[\phi]}=L^r(m)$. Moreover, $\big(L^{r}(m),L^{\infty}(m)\big)_{[\gamma]}=L^q(m)$. The space of compactly supported functions with vanishing integral is contained in $H^1(m)\cap L^{\infty}(m)$ and is dense in $L^r(m)$ and $L^q(m)$. Then by the reiteration theorem \cite[Th. 2]{wolff}, we deduce that $\big(H^{1}(m),L^{\infty}(m)\big)_{[\theta]}=L^q(m)$.  
 Property $(iii)$ follows from $(i)$ and $(ii)$ and the fact that $H^1(m)\subset L^1(m)$ and $L^{\infty}(m)\subset BMO(m)$. 
 \end{proof} 
% 
% 
%{\color{red}COMMENTO CEH NON METTEREI NELL'ARTICOLO:
%
%Let $C_c(V)$ be the set of compactly supported functions on $V$ and 
%$$ C_{c,0}(V)=\Big\{f\in C_c(V): \int f dm=0\Big\}\,.$$
%Take $g\in C_c(V)$. For every $n\in \mathbb N$ define $B_n=B(o,n)$ and 
%$$
%g_n=g\chi_{B_n}- m(B_n)^{-1}\int_{B_n}g\,dm\,\chi_{B_n}\,.
%$$
%Then $g_n\in C_{c,0}(V)$ and for every $q\in (1,\infty)$
%$$
%\|g-g_n\|_q \leq \Big(\int_{B_n^c}|g|^q\,dm\Big)^{1/q}+\|g\|_{1}m(B_n)^{-1+1/q}\rightarrow 0 \,,
%$$
%which tends to $0$ when $ n\rightarrow +\infty$. 
%
%This shows that $C_{c,0}(V)$ is dense in $C_c(V)$ with respect to the $L^q$-norm. Hence $C_{c,0}(V)$ is dense in $L^q(m)$ for every $q\in (1,\infty)$. 
%} 
% 
% 
% 

\subsection{Integral operators}

We now prove that  integral operators which satisfy a H\"ormander's condition are bounded from  $H^1(m)$ to $L^1(m)$  and from $L^\infty(m)$  to $BMO(m).$ 
\begin{thm}\label{t: intop}
Let $T$ be a linear operator which is bounded on $L^2(m)$ and admits a
kernel $K$.
\begin{itemize}
    \item[(i)] Assume $K$ satisfies the condition
\begin{align}\label{1star}
  \sup_{R \in \mathcal{R}} \sup_{y,z \in R} \int_{(R^*)^c} |K(x,y)-K(x,z)| \ dm(x) < + \infty,
\end{align}
where, for any $R=R_{h'}^{h''}(x) \in \mathcal{R}$, we define $R^{*}=\{x \in V : d(x,R)<h' \}$. Then $T$ extends to a bounded operator from $H^1(m)$ to $L^1(m)$ and on $L^p(m)$, for $1<p<2.$ 
\item[(ii)]
If  K satisfies the condition 
\begin{align}\label{duestar}
     \sup_{R \in \mathcal{R}} \sup_{y,z \in R} \int_{(R^*)^c} |K(y,x)-K(z,x)| \ dm(x) < + \infty,
\end{align}
where $R^*$ is defined as in $(i)$, then $T$ extends to a bounded operator from $L^\infty(m)$ to $BMO(m)$ and on $L^q(m)$, for $2<q<+\infty$.
\end{itemize}
\end{thm}
\begin{proof}
We first observe that given $R=R_{h'}^{h''}(x) \in \mathcal{R}$, we have $m(R^{*})= (h''+h'-1)m(x) \le 3 m(R).$ Thus we can follow verbatim \cite[Th. 3]{ATV1}  and conclude that, if $K$ satisfies \eqref{1star}, then $T$ is bounded from $H^1(m)$ to $L^1(m)$. By Theorem \ref{realintH1Lp2}, it follows that $T$ is bounded on $L^p(m)$, for $1<p<2$. Suppose that $K$ satisfies \eqref{duestar}. Then the kernel $K^*$ of the adjoint operator $T^*$ satisfies \eqref{1star}. By $(i)$, $T^*$ extends to a bounded operator from $H^1(m)$ to $L^1(m)$ and on $L^p(m)$, for $1<p<2.$ By duality it follows that $T$ extends to a bounded operator from $L^\infty(m)$ to $BMO(m)$ and on $L^q(m)$, for $2<q<+\infty$.
\end{proof}

\begin{oss} Theorem \ref{t: intop} applies to suitable spectral multipliers and to first order Riesz transforms associated with a distinguished Laplacian on the homogeneous tree (see \cite[Th. 2.3]{hs} and \cite{ATV1}). We are currently studying the boundedness properties of spectral multipliers and Riesz transforms on more general trees and for more general Laplacians which are self-adjoint with respect to a flow measure. This is work in progress. 
\end{oss}

\bigskip
	
	{\bf{Acknowledgments.}} 	
	Work partially supported by the MIUR project ``Dipartimenti di Eccellenza 2018-2022" (CUP E11G18000350001) and the Progetto GNAMPA 2020 ``Fractional Laplacians and subLaplacians on Lie groups and trees". 
	
	The authors are members of the Gruppo Nazionale per l'Analisi Matema\-tica, la Probabilit\`a e le loro Applicazioni (GNAMPA) of the Istituto Nazionale di Alta Matematica (INdAM).

\bibliographystyle{acm}
{\small
\bibliography{AFMM-vf}}

\end{document}